\documentclass{elsarticle}
\usepackage[left=1.5cm, right=2.5cm, top=3cm, bottom=3 cm]{geometry}
\usepackage{hyperref}
\usepackage{amsmath}
\usepackage{amsfonts}
\usepackage{amsthm}
\usepackage{amsmath}
\usepackage{pgfplots}
\usepackage{pgfplotstable}
\usepackage{subcaption}
\usepackage{tikz}
\usepackage{amssymb}
\usepackage{multirow}
\usepackage{lipsum}
\usepackage{amsfonts}
\usepackage{graphicx}
\usepackage{epstopdf}
\usepackage{algorithmic}
 \usepackage{mathrsfs}
 
\usepackage{slashbox}
\usepackage{booktabs}

\usepackage{todonotes}
\newtheorem{thm}{Theorem}
\newtheorem{lem}[thm]{Lemma}
\newtheorem{corollary}[thm]{Corollary}
\newdefinition{rmk}{Remark}
\definecolor{seabornblue}{rgb}{0.2980392156862745, 0.4470588235294118, 0.6901960784313725}
\definecolor{seaborngreen}{rgb}{0.3333333333333333, 0.6588235294117647, 0.40784313725490196}
\definecolor{seabornred}{rgb}{0.7686274509803922, 0.3058823529411765, 0.3215686274509804}
\journal{Journal of \LaTeX\ Templates}

\usepackage{todonotes}

\newcommand{\revise}[1]{#1}
\newcommand{\revisenew}[1]{\textcolor{black}{#1}}
\DeclareMathOperator{\grad}{grad}
\DeclareMathOperator{\Vol}{Vol}

\begin{document}

\begin{frontmatter}

\title{A new mixed finite-element method for $H^2$ elliptic problems\tnoteref{mytitlenote}}
\tnotetext[mytitlenote]{The work of AH and SM was partially supported by an NSERC discovery grant. The work of PEF is supported by EPSRC grants EP/R029423/1 and EP/V001493/1. The authors would like to thank an anonymous referee for substantially improving the manuscript.}


\author[mymainaddress]{Patrick E. Farrell}
\ead{patrick.farrell@maths.ox.ac.uk}

\author[mysecondaryaddress]{Abdalaziz Hamdan\corref{mycorrespondingauthor}}
\cortext[mycorrespondingauthor]{Corresponding author}
\ead{ahamdan@mun.ca}
\author[mysecondaryaddress]{Scott P. MacLachlan}
\ead{smaclachlan@mun.ca}

\address[mymainaddress]{Mathematical Institute, University of Oxford, Oxford, UK }
\address[mysecondaryaddress]{Department of Mathematics and Statistics, Memorial University of Newfoundland, St.~John's, NL, Canada }

\begin{abstract}
Fourth-order differential equations play an important role in many applications in science and engineering. In this paper, we present a three-field mixed finite-element formulation for fourth-order problems, with a focus on the effective treatment of the different boundary conditions that arise naturally in a variational formulation. Our formulation is based on introducing the gradient of the solution as an explicit variable, constrained using a Lagrange multiplier. The essential boundary conditions are enforced weakly, using Nitsche's method where required.
As a result, the problem is rewritten as a saddle-point system, requiring analysis of the resulting finite-element discretization and the construction of optimal linear solvers. Here, we discuss the analysis of the well-posedness and accuracy of the finite-element formulation. Moreover, we develop monolithic multigrid solvers for the resulting linear systems. Two and three-dimensional numerical results are presented to demonstrate the accuracy of the discretization and efficiency of the multigrid solvers proposed. 
\end{abstract}


\end{frontmatter}

\section{Introduction}
Fourth-order differential operators often appear in mathematical models of thin films and plates \cite{MR1915664,pevnyi2014modeling,MR899701}, and pose
significant challenges in numerical simulation in comparison to equations governed by more familiar second-order operators. A motivating example arises from modeling equilibrium states of smectic A liquid crystals (LCs), which correspond to minimizers of a given energy functional. 
%The energy functionals for smectic (and other) LCs involve various terms that appear in ``competition'', and the overall extrema typically don\textquotesingle t arise from extremizing each term, since making one term small likely makes another large. Thus, much of the study of liquid crystals is focused on understanding ``geometric frustration'', where boundary conditions force LCs into configurations that are far from globally optimal, but represent ``the best possible'' configurations when constrained by boundaries or other terms.
For example, the Pevnyi, Selinger, and Sluckin energy functional for smectic A liquid crystals is given by~\cite{pevnyi2014modeling}:
\begin{equation}\label{1.1}
I(u,\vec{n})=\int_{\Omega}^{ }\frac{a}{2}u^2+\frac{b}{3}u^3+\frac{c}{4}u^4+B[(\partial_i\partial_j+q^2n_in_j)u]^2+\frac{K}{2}(\partial_in_j)^2,
\end{equation} 
where $\Omega\subset \mathbb{R}^d,\ d\in\{2,3\}$ is a bounded Lipschitz domain, $a,b,q, B$, and $K$ are positive real-valued constants determined by the experiment and material under consideration, $\vec{n}:\Omega\rightarrow \mathbb{R}^d$ is a unit vector field called the director, and $u:\Omega\rightarrow \mathbb{R}$ is the smectic order parameter representing the density variation of the LC. This energy is to be minimized subject to the constraint that $\vec{n}\cdot\vec{n} = 1$ pointwise almost everywhere. When enforcing this constraint with a Lagrange multiplier, the Euler--Lagrange equations for \eqref{1.1} lead to a coupled system of PDEs, with a fourth-order operator applied to $u$, a second-order operator acting on $\vec{n}$, and an algebraic constraint.

Motivated by such examples, several families of finite-element methods have been developed to approximate solutions of PDEs with fourth-order terms. In this work, we consider the minimization of a simplified form of the energy \eqref{1.1} with suitable boundary conditions, given in variational form as
\begin{equation} \label{eqn:our_problem}
\min_{v\in H^2(\Omega)}\frac{1}{2}\int_{\Omega}(\Delta v)^2+c_0\nabla v \cdot\nabla v+c_1v^2-\int_{\Omega}fv,
\end{equation}
with nonnegative constants $c_0$ and $c_1$. While the variational formulation in \eqref{1.1} is written in terms of the Hessian operator, here we consider the classical fourth-order biharmonic (Laplacian squared) and will consider the Hessian problem in future work.  Sufficiently smooth extremizers of \eqref{eqn:our_problem} must satisfy its Euler--Lagrange equations, which yield a fourth-order problem,
\begin{equation}\label{eq:mod_biharmonic}
\Delta^2u-c_0\Delta u + c_1u=f.
\end{equation}
We consider three-field mixed formulations for this fourth-order problem, with a particular focus on the treatment of the boundary conditions that arise naturally from the transition from the variational to strong forms. These formulations introduce the gradient of the solution as an explicit variable constrained using a Lagrange multiplier. Our approach is general in the sense that we are able to use elements of order $k, \ k+1,\ \text{and}\ k+1$ for the solution, its gradient, and the Lagrange multiplier respectively, where $k$ can be as large as the smoothness of the solution allows. The existence and uniqueness proofs are not complicated. A drawback here is that our formulation provides suboptimal convergence for some boundary conditions, as discussed below. 

If $c_0=c_1=0$, then \eqref{eqn:our_problem} represents the classical  
biharmonic equation. Many different types of finite-element methods have been considered in this context.
Conforming methods, in which the finite-dimensional space is a
subspace of the Sobolev space $H^2(\Omega)$, rely on the use of complicated basis functions. These require a high number of degrees of freedom per element, especially in three dimensions. Moreover, the elements are typically not affine equivalent; i.e.~the basis functions cannot be mapped to each element using a reference element in the standard way, and more complicated approaches are needed \cite{MR0520174,kirby2012common, articl}. In order to avoid the use of such $C^1$ elements, other types of finite elements can be used, leading to nonconforming methods in which the finite-element space is not a subspace of $H^2(\Omega)$, such as Morley and cubic Hermite elements \cite{MR2373954,MR2207619,MR0520174}. \revise{While simple to define, these elements are generally difficult to implement in existing software, and also require analysis of the consistency error, which sometimes implies suboptimal convergence when the consistency error is larger than the interpolation error of the element~\cite{MR3664257}.}

$C^0$ interior penalty (C0IP) methods can also be used for fourth-order problems, where the continuity of the function derivatives are weakly enforced using stabilization terms on interior edges~\cite{MR2142191,MR3022211,MR345432}. Brenner and Sung \cite{ MR2142191} solved the problem $\Delta^2 u-\nabla \cdot(\beta(x)\nabla u)=f$, where $\beta(x)$ is a nonnegative $C^1$ function. Their approach is to find $u\in CG_k(\Omega,\tau_h)$, $k\geq 2$, that satisfies the system
\begin{equation}\label{C0IP}
a_h(u,\phi)+b_h(u,\phi)+\gamma c_h(u,\phi)=\langle f,\phi\rangle, \ \forall \phi\in CG_{k}(\Omega,\tau_h),
\end{equation} 
where $\gamma>0$ is a penalty parameter, $CG_{k}(\Omega,\tau_h)$ is the space of continuous Lagrange elements of degree $k$ on a triangulation $\tau_h$ of the domain $\Omega$, and    
\begin{eqnarray}
a_h(u,\phi)&=& \sum_{T\in\tau_h}\int_{T}\left(\nabla\nabla u:\nabla\nabla \phi+\beta(x)\nabla u\cdot\nabla \phi\right),\\
b_h(u,\phi)&=& \sum_{e\in\epsilon_h}\int_{e}\left\{\frac{\partial^2 u }{\partial^2n}\right\}\left[\frac{\partial \phi}{\partial n}\right]+  \sum_{e\in\epsilon_h}\int_{e}\left\{\frac{\partial^2 \phi }{\partial^2n}\right\}\left[\frac{\partial u}{\partial n}\right],\\
c_h(u,\phi)&=&\sum_{e\in\epsilon_h}\frac{1}{|e|}\int_{e}\left[\frac{\partial u}{\partial n}\right]\left[\frac{\partial \phi}{\partial n}\right],
\end{eqnarray}
where $\left\{\frac{\partial u}{\partial n}\right\}$ and $\left[\frac{\partial \phi}{\partial n}\right] $ denote the standard average and jump on each edge. Here, $\tau_h$ is the set of cells in a mesh and $\epsilon_h$ is the set of edges.  While C0IP methods have advantages, such as enabling the use of simple Lagrange elements and the ability to use arbitrarily high-order elements \cite{MR2142191}, they also have some disadvantages. The weak forms are more complicated than those used for classical conforming and nonconforming methods. Moreover, the need for the penalty parameter is also a drawback, as it is sometimes not trivial to decide how large this parameter must be to achieve stability, especially as parameters in the PDE are varied~\cite{nguyen2008discontinious}. Similarly, discontinuous Galerkin approaches can also be applied to this problem \cite{MR0431742,MR2298696}, augmenting the forms in \eqref{C0IP} to account for basis functions that do not enforce $C^0$ continuity across elements. These share the disadvantages of C0IP methods, while requiring more degrees of freedom than $C^0$ approaches.
%Here $n_e$ is the normal unit vector corresponding to the edge $e\subset \Omega$, of arbitrary orientation, and pointing from the triangle $T_{-}$ to $T_{+}$, and be the outward unit normal to $\Omega$ if $e\subset \partial\Omega$.
%The jump of the normal derivative of $v$, $\big[\frac{\partial v}{\partial n}\big]$,  is defined as 

%\[   
%\left[\frac{\partial v}{\partial n}\right] = 
%\begin{cases}
%\frac{\partial v_{T_+}}{\partial n_e}|_e-\frac{\partial v_{T_-}}{\partial n_e}|_e &\quad\text{if } e\subset \Omega,\\
% -\frac{\partial v}{\partial n_e}&\quad\text{if}\  e\subset \partial\Omega,\\
%\end{cases}
%\]\\
%and the average
%\[   
%\left\{\frac{\partial v}{\partial n}\right\} = 
%\begin{cases}
%\frac{1}{2}\left[\frac{\partial^2v_{T_{+}}}{\partial n_e^2}+\frac{\partial^2v_{T_{-}}}{\partial n_e^2}\right] &\quad\text{if } e\subset \Omega,\\
%\frac{\partial^2v}{\partial n_e^2}&\quad\text{if}\  e\subset \partial\Omega.\\
%\end{cases}
%\]\\

% approach to the (C0IP) approaches is called the discrete galerkin\cite{MR2416975}.  This approach uses discontinuous finite elements to approxim\volate the solution of the two-dimensional biharmonic problem. Note that the (C0IP) approaches use fewer degrees of freedom because of the shared nodal values on the interior edges\cite{MR3873982}.

Another attractive option to avoid using $H^2$-conforming methods is mixed finite-element methods, in which the gradient or the Laplacian of the solution are approximated in addition to the solution itself \cite{MR0520174,MR3580405,cheng2000some,li2017stable,behrens2011mixed,malkus1978mixed,banz2017two, MR3667017, MR583486,MR3533246,MR4113080}. A natural classification of such mixed finite-element methods is based on how many functions (fields) are directly approximated. Given clamped boundary conditions, where both $u$ and $\nabla u$ are prescribed on the boundary, two functions are approximated in \cite{MR0520174, cheng2000some, malkus1978mixed, MR583486}, both $u$ and either its gradient or its Laplacian/Hessian. In \cite{MR0520174}, the biharmonic problem is rewritten as a coupled system of Poisson equations, in which the unknown and its Laplacian are both directly approximated. In \cite{malkus1978mixed}, the 2D biharmonic problem is approximated by minimizing
\begin{equation*}
J(u,\vec{v})=\frac{1}{2}\|\nabla\vec{v}\|_0^2+\frac{1}{2\epsilon}\|\rho_0(\vec{v}-\nabla u)\|_0^2-\langle f,u \rangle, \quad \text{for} \ 0<\epsilon\leq ch^2,
\end{equation*}
where $\rho_0$ is the orthogonal projection from $[L^2(\Omega)]^2$ to the space of piecewise constant functions, and the functions $u$ and $\vec{v}$ are approximated using bilinear elements on a rectangular mesh. An error analysis of this method requires the solution to be at least in $H^{4.73}(\Omega)$~\cite{MR645657}. A similar approach solves the $d$-dimensional biharmonic problem, replacing the $L^2$ projection onto piecewise constant functions with that onto the space of multilinear vector-valued functions whose $i^{th}$ component is independent of $x_i$ \cite{cheng2000some}. This approach requires less regularity on the solution, $u\in H^4(\Omega)$, than was required in \cite{malkus1978mixed}. These approaches only treat clamped boundary conditions.

% In this approach, the condition $\psi=\nabla v$ is enforced via the projected penalty term.  In \cite{cheng2000some}, the authors propose a mixed finite-element formulation for the biharmonic equation on a rectangular mesh, where the domain $\Omega$ is partitioned to a finite number of a regular rectangles whose sides are parallel to the coordinate axis. They then define $S_h=\{v_h\in H_0^1(\Omega): v_K|_T\in Q_1(K)\}$, where
%$K$ is a rectangle, and $\boldsymbol{Q_h}$=$\{\mu \in [L^2(\Omega)]^2:\mu|_K\in\prod_{i=1}^{2}Q_1^{(i)}(K)\}$, where $Q_1$ is the space of multilinear functions, and $Q_1^{(i)}$ is the subspace of $Q_1$ defined by functions that are independent of $x_i$. For example, in $2D$,  the first component of $\boldsymbol{Q_h}$ is linear in $y$ but independent of $x$, and the second component is linear in $x$ but independent of $y$. The biharmonic equation is approximated via the following minimization problem \cite{cheng2000some}
%\begin{equation}
%J(u_h,\phi_h)=\inf \ J(v_h,\psi_h), \ \ (v_h,\psi_h)\in S_h\times [S_h]^2,
%\end{equation}
%where
%\begin{equation*}
%J(v,\psi)=\frac{1}{2}\|\nabla\psi\|_0^2+\frac{1}{2\epsilon}\|\rho_1(\psi-\nabla v)\|_0^2-(f,v),
%\end{equation*}
%and $\rho_1:[L^2(\Omega)]^2\rightarrow\boldsymbol{Q_h}$
%is the orthogonal projection onto $\boldsymbol{Q_h}$.
A second class of mixed finite-element methods is that of four-field formulations, in which $u,\ \nabla u, \ \nabla^2u, \ \text{and}\  \nabla\cdot(\nabla^2 u)$ are directly approximated. In \cite{li2017stable}, a mixed formulation approximating these fields and its stability in $H^1_0(\Omega)\times [{H}^1_0(\Omega)]^2\times \boldsymbol{L}^2_{\text{sym}}(\Omega)\times H^{-1}(\text{div}, \Omega)$ is discussed for $\Omega\subset\mathbb{R}^2$, where $\boldsymbol{L}^2_{\text{sym}}(\Omega)$ is the space of $2\times 2$ symmetric tensors with components in $L^2(\Omega)$, and $H^{-1}(\text{\text{div}}; \Omega)$ is the dual space of $H_0(\text{rot},\Omega)=\{\vec{\psi}\in [L^2(\Omega)]^2 \ |\ \text{rot}\,\vec{\psi}\in L^2(\Omega), \ \vec{\psi}\cdot\vec{t}=0\ \text{on } \partial\Omega\}$, where $\vec{t}$ is the unit tangent vector to $\partial\Omega$. A similar approach with different function spaces is given in \cite{behrens2011mixed}. This approach, focused on the discrete level, finds $(u_h,\vec{q}_h,\ \bar{ z}_h, \ \vec{\sigma}_h ) \in DG_k(\Omega,\tau_{h})\times [{DG}_k(\Omega,\tau_{h})]^2\times \revise{\boldsymbol{RT}_{k+1}(\Omega,\tau_h)\times RT_{k+1}(\Omega,\tau_{h})}\subset L^2(\Omega)\times[{L}^2(\Omega)]^2\times \boldsymbol{H}(\text{div}, \Omega) \times H(\text{div};\Omega) $, where $u_h,\vec{q}_h,\ \bar{ z}_h, \ \text{and}\  \vec{\sigma}_h$ are approximations of  $u,\ \nabla u, \ \nabla^2u, \ \text{and} \ \nabla\cdot(\nabla^2 u)$ respectively, and $\bar{y}\in \boldsymbol{H}(\text{div};\Omega)) $ means that each row of the tensor
$\bar{y}$ belongs to $H(\text{div}; \Omega)$. \revise{As defined in Section \ref{sec:background}}, ${DG}_k(\Omega,\tau_{h})$ and $RT_k(\Omega,\tau_{h})$ denote the discontinuous Lagrange and Raviart--Thomas approximation spaces of order $k$ on mesh $\tau_{h}$, respectively, \revise{and $\boldsymbol{RT}_k(\Omega,\tau_h)$ denotes the} tensor-valued functions with rows in $RT_k(\Omega,\tau_{h})$. However, these four-field formulations lead to discretizations
with large numbers of degrees of freedom, posing difficulties in the development of efficient linear solvers. 

\revise{The third class of mixed finite-element methods is that of three-field formulations~\cite{MR3580405,MR3667017,MR3533246,MR4113080}.  In \cite{MR3667017}, Gallistl presents mixed finite-element methods for general polyharmonic problems. As a special case, a continuum three-field formulation for the biharmonic is proposed in \cite[Section 4.1]{MR3667017}, based on decomposing the problem into two Poisson-type equations and a generalized Stokes problem. However, when discretized, this formulation requires a fourth field, due to the need for a Lagrange multiplier to enforce a rot-free condition on one of the three fields. These formulations are shown to be effective for both clamped and simply-supported boundary conditions. Another family of three-field mixed formulations is based on the Helmholtz decomposition of the space $H^{-1}(\text{div} \, \boldsymbol{\textbf{div}},\Omega)_{\text{sym}}=\left\{\mathscr{M}\in \boldsymbol{L}^2_{\text{sym}}(\Omega)| \ \text{div} \, \boldsymbol{\textbf{div}} \mathscr{M}\in H^{-1}(\Omega)  \right\}$, introduced in \cite{MR3533246}. Finally, \cite{MR3580405} proposes a methodology where the unknowns are the function, its gradient, and a Lagrange multiplier.} Assuming again homogeneous clamped boundary conditions, these lead to finding the saddle-point $(u,\vec{v}, \vec{\alpha})\in H^1_0(\Omega) \times H_0(\text{div};\Omega)\times M$ of the Lagrangian functional 
\begin{equation}
	\mathcal{L}\big((u,\vec{v}), \vec{\alpha}\big)=\frac{1}{2}\|\nabla\cdot\vec{v}\|_0^2+\int_{\Omega}\vec{\alpha}\cdot(\vec{v}-\nabla u)-\int_{\Omega}fu
\end{equation}
%\begin{eqnarray}
%a\big((u,\vec{v}), (\phi,\vec{\psi})\big)+ b\big((\phi,\vec{\psi}), \vec{\alpha}\big)&=&F\big((\phi, \vec{\psi})\big), \quad \forall (\phi,\vec{\psi}) \in H_0(\Omega) \times H_0(\text{div};\Omega), \\
%b\big((u,\vec{v}), \vec{\beta}\big)&=&0, \quad \forall \vec{\beta}\in M,
%\end{eqnarray} 
where 
%\begin{eqnarray}
%a\big((u,\vec{v}), (\phi, \vec{\psi})\big)&=&\int_{\Omega}\nabla \cdot \vec{v}\,  \nabla\cdot\vec{\psi}+\int_{\Omega}(\vec{v}-\nabla u)\cdot(\vec{\psi}-\nabla \phi)\\
%b\big((u,\vec{v}), \vec{\alpha}\big)&=& \int_{\Omega}\vec{\alpha}\cdot (\vec{v}-\nabla u)\\
%F\big((\phi, \vec{\psi})\big)&=& \int_{\Omega}f\phi, 
%\end{eqnarray}
$M=\{\vec{\alpha}\in H_0(\text{div};\Omega)^* \ |\ \nabla\cdot\vec{\alpha}\in H^{-1}(\Omega)\}$. \revise{Here, $H_0(\text{div};\Omega)^*$ and $H^{-1}(\Omega)$ are the dual spaces of $H_0(\text{div};\Omega)$ and $H^1_0(\Omega)$ respectively, and $H_0(\text{div}; \Omega):=\{\vec{v}\in H(\text{div};\Omega) \ |\  \vec{v}\cdot\vec{n}=0 \text{ on } \partial \Omega\}$.} At the discrete level, the method in \cite{MR3580405} finds $(u_h,\vec{v}_h,\ \vec{\alpha}_{2h})\in CG_{1}(\Omega,\tau_h)\times RT_{1}(\Omega,\tau_{h})\times DG_{0}(\Omega,\tau
_{2h})$, where the Lagrange multiplier $\vec{\alpha}_{2h}$ is constructed in $\tau_{2h}$ to guarantee well-posedness at the discrete level and to achieve an optimal error estimate. This approach only treats clamped boundary conditions, and requires the use of different meshes in the discretization. Moreover, it is not mentioned if the discretization can be generalized to higher orders. Here, we propose a similar three-field formulation, but treating the generalized problem in \eqref{eq:mod_biharmonic} with more general boundary conditions, and using different discretization spaces of arbitrarily high degree. Unlike conforming methods, our approach works effectively in both two and three dimensions.

Strongly imposing essential boundary conditions with some finite-element basis functions is difficult \cite{articl}. In addition, it can, sometimes, negatively affect properties of the finite-element method, such as its stability  
and accuracy~\cite{aziz,MR2501054}. Weakly imposing the boundary conditions via a penalty method \cite{MR351118,MR0478662} may help. An attractive family of penalty methods are the Nitsche-type methods \cite{Nitsche_1971} for which optimal convergence can be achieved. Applications of Nitsche's method to second-order PDEs can be found in \cite{MR2501054,article,aziz}. Moreover, Nitsche-type penalty methods have been used to impose essential boundary conditions for some discretizations of the biharmonic and other fourth-order problems \cite{MR2683379,articl,BENZAKEN2021113544}.  While we are able to impose a variety of boundary conditions directly in our variational formulation, we utilize Nitsche-type penalty methods for a particular case where strong enforcement of the boundary conditions leads to problems establishing inf-sup stability of the discretization.

At the discrete level, the resulting linear system of our three-field formulation is a saddle point system~\cite{MR2168342}, of the form
\[
\begin{bmatrix}
A&B^T\\
B&0
\end{bmatrix}
\begin{bmatrix}
U
\\{\alpha}
\end{bmatrix}=
\begin{bmatrix}
f\\g
\end{bmatrix}
\]
where $U$ represents discrete degrees of freedom associated with both $u$ and $\vec{v}$, while $\alpha$ represents discrete degrees of freedom associated with $\vec{\alpha}$, leading to matrices  $A\in \mathbb{R}^{n\times n}, \ B\in \mathbb{R}^{m\times n}$ and the zero matrix $0 \in \mathbb{R}^{m \times m}$. In our formulation, $A$ will be symmetric and positive semi-definite. This kind of problem appears in many areas of computational science and engineering \cite{MR2168342}. For discretized PDEs, the condition number of such systems usually grows like $h^{-k}$ for $k > 0$, resulting in increasingly ill-conditioned systems as the mesh size, $h$, goes to zero. This growth of the condition number leads to slow convergence of unpreconditioned Krylov methods. Therefore, we employ preconditioning in order to develop a mesh-independent algorithm to solve these systems. Two common families of preconditioners are block factorization \cite{MR2155549,MR4016136,MR1762024} and monolithic multigrid preconditioners~\cite{MR3639325,MR3504546,adler2021,MR848451}.  In this work, we propose an effective monolithic multigrid solvers for the arising saddle-point systems \cite{adler2021,MR3639325,Pcpatch}. \revise{We note that efficient multigrid solvers have also been investigated for other discretization approaches, including C0IP methods~\cite{MR2217379,10.1007/978-3-642-35275-1_13} and mixed methods using Hellan-Herrmann-Johnson elements~\cite{MR3533246,MR3817803}.}

This paper is organized as follows. In Section \ref{sec:background}, a brief summary is given of the Sobolev and finite-element spaces employed. The weak forms, uniqueness of solutions at the continuum and discrete levels, and an error analysis are presented in Sections \ref{sec:continuum} and \ref{sec:discrete}. The monolithic multigrid preconditioner and the details of the linear solver are presented in Section \ref{sec:multigrid}. Finally, numerical experiments showing the accuracy of the finite-element method and the effectiveness of the linear solver are given in Section \ref{sec:numerical}.          

\section{Background}
\label{sec:background}

\revise{Throughout this paper, we assume $\Omega\subset \mathbb{R}^d$, $d\in\{2,3\}$ to be an open, connected and  bounded polytope with Lipschitz boundary.} On a simplex $T \in \tau_h$, all degrees of freedom of the discontinuous Lagrange $DG_k(\Omega,\tau_h)$ element, $k\geq 0$, are considered to be internal; i.e., no continuity is imposed by these elements \cite{kirby2012common}. In contrast, the continuous Lagrange $CG_k(\Omega,\tau_h)$ elements, $k\geq 1$, possess full $C^0$ continuity across element edges. Here, we primarily make use of $DG_k(\Omega,\tau_h)$ approximations of functions in $L^2(\Omega)$. We also consider the Raviart-Thomas $RT_{k}(\Omega,\tau_{h})$ element, $k\geq 1$, which is $H(\text{div})$-conforming, where the normal component is continuous across element faces, \revise{defining the function space $RT_k(\Omega,\tau_h)|_T\in[P_{k-1}(T)]^d + \vec{x}P_{k-1}(T)$, $\forall T\in \tau_h$, where $P_{k-1}(T)$ is the space of multivariate polynomials of degree at most $k-1$ on simplex $T$. We point out that, while the lowest-order Raviart-Thomas element is sometimes denoted $RT_{0}(\Omega,\tau_h)$, we follow the alternate notation (cf.~\cite{kirby2012common}),  where the lowest-order element is denoted $RT_{1}(\Omega,\tau_h)$, with the property that $RT_{k}(\Omega,\tau_h)\subset DG_{k}(\Omega,\tau_h)$. }
Finally, we define $RT^{\Gamma}_{k}(\Omega,\tau_{h})=\left\{\vec{v}\in RT_{k}(\Omega,\tau_{h})\middle| \vec{v}\cdot\vec{n}=0\ \text{on}\ \Gamma\subset\partial\Omega \right\}$. A standard approximation result for these elements is stated next.  
\begin{thm}\label{pro}\cite{MR3097958,MR2322235,kirby2012common}
	Let $I^{k}_{h}:H^{k+1}(\Omega)\rightarrow DG_{k}(\Omega,\tau_h)$, $\varPi^{k}_{h}:[H^{k+1} (\Omega)]^d\rightarrow RT_{k}(\Omega,\tau_h)$, and $L^{k}_{h}:[H^{k+1}(\Omega)]^d\rightarrow  [CG_{k}(\Omega,\tau_h)]^d$ be the finite-element interpolation operators. Then there exist constants $\tilde{c}$, \revise{$c_1$, $c_2$}, and $\hat{c}$, such that for any $u\in H^{k+1}(\Omega)$ and $\vec{v}\in \left[H^{k+1}(\Omega)\right]^d$,
	\begin{align}
	\|u-I^k_hu\|_0&\leq \tilde{c}h^{k+1}|u|_{k+1},\ \ \forall k\geq 0, \\
	\revise{\|\vec{v}-\varPi^k_h \vec{v}\|_{0}}&\revise{\leq c_1h^{k}|\vec{v}|_{k}, \ \ \forall k>0,}\\
\revise{	\|\nabla\cdot\left(\vec{v}-\varPi^k_h \vec{v}\right)\|_{0}}&\revise{\leq c_2h^{k}|\vec{v}|_{k+1}, \ \ \forall k>0,}\\
%	\intertext{and}
	\|\vec{v}-L^k_h \vec{v}\|_1&\leq \hat{c}h^k|\vec{v}|_{k+1},\ \ \forall k>0.
	\end{align}
\end{thm}
%\pef{Why is the \textbf{div} in bold? It's not below.}

%As we will make use of a mixed formulation, our proofs will be aided by recalling a standard result for mixed formulations of the Poisson problem as follows.
%\begin{lem}\label{Lemma1}
%	Given $u\in L^2(\Omega)$, define $\phi$ as the solution of
%	\begin{eqnarray*}
%		\Delta \phi&=& u \ \text{in}\  \Omega,\\
%		\phi&=&0 \ \text{on}\ \Gamma_D, \ \text{and}\ \nabla u\cdot\vec{n}=0 \ \text{on}\ \Gamma_N,   
%	\end{eqnarray*} 	
%	where $\Gamma_D$ is nonempty.  Let $V = H_{\Gamma_N}(\textrm{div};\Omega)$ and $Q = L^2(\Omega)$. The mixed formulation to find $(\vec{\alpha},\phi)\in V\times Q$ such that 
%	\begin{eqnarray}
%\int_{\Omega}\eta\nabla\cdot\vec{\alpha}&=&\int_{\Omega}u\eta,% \ \ \forall \eta \in Q,\label{eq2.4}\\
%\int_{\Omega}\vec{\alpha}\cdot\vec{\beta}+\phi\nabla \cdot\vec{\beta}&=&0, \ \ \forall \vec{\beta} \in V,
%\end{eqnarray}
%is well-posed and, moreover, $\|\vec{\alpha}\|^2_{\text{div}}+\|\phi\|_0^2\leq \Lambda\|u\|_0^2$, where $\Lambda$ is a positive constant that depends on the constants of coercivity, continuity, and the inf-sup conditions.
%\end{lem}
%\begin{rmk}
%	The result above is also true at the discrete level, for $u\in DG_{k}(\Omega,\tau_h)$ and $(V,Q)= RT^{\Gamma_N}_{k+1}(\Omega,\tau_{h})\times DG_{k}(\Omega,\tau_h)$.
%\end{rmk}
%\begin{rmk}\label{Remark2}
%	The choice of $\eta=u$ in Equation \eqref{eq2.4} implies that $\|u\|_0^2=\int_{\Omega}u\nabla\cdot\vec{\alpha}$.
%\end{rmk}
%Elementwise, the space $BDM_{k+1}(T)$ is the smallest space such that

An important property of our discretization is that it benefits from the usual mimetic relationships between $RT_{k+1}(\Omega,\tau_h)$ and $DG_{k}(\Omega,\tau_h)$, summarized in the following results.
\begin{lem}\cite{MR1401938, DNArnold_etal_2000a}\label{lemma,helm}
	Assume $\Omega$ is simply-connected. Then the Helmholtz decomposition of $RT_{k+1}(\Omega,\tau_h)$ is
	\begin{equation}
	RT_{k+1}(\Omega,\tau_h)=\bigg(\nabla\times V_h\bigg)\oplus \bigg(\grad_h DG_{k}(\Omega,\tau_h)\bigg),
	\end{equation} 
	where $\grad_h:DG_{k}(\Omega,\tau_h)\rightarrow RT_{k+1}(\Omega,\tau_h)$ is the discrete gradient operator, defined by
	\begin{equation*}
	\int_{\Omega}\grad_h u\cdot\vec{v}=-\int_{\Omega}u\nabla\cdot\vec{v}, \ \ \forall \vec{v}\in RT_{k+1}(\Omega,\tau_h).
	\end{equation*} 
	For $d=2$, $\nabla\times =
	\begin{bmatrix}
	-\frac{\partial}{\partial y}   \\
	\frac{\partial}{\partial x}\\
	\end{bmatrix}$ and $V_h=CG_{k+1}(\Omega,\tau_{h})$, while $V_h=N^1_{k+1}(\Omega,\tau_{h})$ for $d=3$, where $N^{1}_{k+1}(\Omega,\tau_{h})$ is the N\'ed\'elec element of the first kind of order $k+1$. 
\end{lem}
\begin{rmk}\cite{kirby2012common,MR3097958}\label{1234}
	$\forall  \vec{v}\in RT_{k+1}(\Omega,\tau_h)$, we have $\nabla\cdot\vec{v}\in DG_{k}(\Omega,\tau_h)$.
\end{rmk}\label{12ttt}
While we largely make use of the standard Sobolev norms, we will also use the ``strengthened'' norm,
\begin{equation} \label{eq:strengthened_norm}
\|\vec{v}\|^2_{{\text{div},{\Gamma}}}= \|\vec{v}\|^2_{\text{div}}+h\|\nabla\cdot\vec{v}\|_{0,\Gamma}^2+\frac{1}{h}\|\vec{v}\cdot\vec{n}\|^2_{0,{\Gamma}} 
\end{equation}
where ${\Gamma}\subset \partial\Omega$ (to be specified later), and
\begin{equation*}
\|\vec{v}\cdot\vec{n}\|^2_{0,{\Gamma}}=\int_{{\Gamma}}|\vec{v}\cdot\vec{n}|^2, \ \ \ \|\nabla\cdot\vec{v}\|^2_{0,{\Gamma}}=\int_{{\Gamma}}|\nabla\cdot\vec{v}|^2 .
\end{equation*}
For these norms, the inverse trace inequality below is a useful result.
\begin{thm}\cite{MR1986022,MR2431403}
	Let $T\in\tau_h$ be a $d$-simplex of $\Omega\subset \mathbb{R}^d$, $d\in\{2,3\}$. Then, for all $u\in DG_k(T)$,
	\begin{equation}\label{2.7}
	\|u\|_{0, \partial T}^2\leq \frac{(k+1)(k+d) \ \Vol_{d-1} (\partial T)}{d \ \Vol_d(T) }\|u\|_{0,T}^2  ,
	\end{equation}
	where $\|u\|_{0, \partial T}^2$ is defined as $\|u\|_{0, \partial T}^2=\int_{\partial T}u^2$, \revise{and $\Vol_{d}(T)$ and $\Vol_{d-1}(\partial T)$ are the Lebesgue measures of $T$ in $\mathbb{R}^d$ and $\partial T$ in $\mathbb{R}^{d-1}$, respectively. }
\end{thm}
\begin{corollary}\label{cor6}
	Consider a triangulation $\tau_h$ of the domain $\Omega \subset \mathbb{R}^d$, and let $\partial \tau_h:=\left\{T\in\tau_h\ \middle| \ \partial T\cap\partial\Omega \neq\varnothing \right\}$. Then,
	\begin{align*}
	\forall u_h\in DG_{k}(\Omega,\tau_h), \|u_h\|_{0,\partial\Omega}^2&\leq \gamma_1(k,\tau_h)\|u_h\|_0^2 \\
	\forall \vec{v}_h\in RT_{k+1}(\Omega,\tau_h),  \|\vec{v}_h\cdot\vec{n}\|^2_{0,\partial\Omega}&\leq \gamma_1(k+1,\tau_h)\|\vec{v}_h\|_0^2,
	\end{align*}
	where 
	\begin{equation}
	\gamma_1(k,\tau_h)=\max_{T\in \partial\tau_h}\frac{(k+1)(k+d) \ \Vol_{d-1}(\partial T)}{d \ \Vol_{d}(T) }.
	\end{equation}
%	and
%	\begin{equation}
%	\gamma_2(k,\tau_h)=\max_{T\in \partial\tau_h}\frac{(k+2)(k+d+1) \ \text{Vol}(\partial T)}{d \ \text{Vol}(T) }.
%	\end{equation} 
\end{corollary}
%\begin{proof}
%Summing \eqref{2.7} over $T\in \partial \tau_h$ and using the facts that
%\begin{equation*}
%\|u_h\|^2_{0,\partial\Omega}\leq\sum_{T\in\partial\tau_h}\|u_h\|^2_{0,\partial T}, \ \text{and} \ \ \ \sum_{T\in\partial\tau_h}\|u\|^2_{0,T}\leq \|u\|_0^2
%\end{equation*}
%completes the proof for $u_h\in DG_{k}(\Omega,\tau_{h})$. Note that the result also holds for the normal component of $\vec{v}_h\in RT_{k+1}(\Omega,\tau_h)\subset \left[DG_{k+1}(\Omega,\tau_{h})\right]^2$.   
%\end{proof}
%\pef{Since this corollary is new, should we include a proof, especially of the second inequality? Maybe in an appendix?}
%For any uniform right-triangular mesh of the unit square, as shown in
%Figure \ref{erd}, on page 15, left, we have
%$\gamma_1(k,\tau_h)=(k+1)(k+2)(2+\sqrt{2})/h$, and
%$\gamma_2(k,\tau_h)=(k+2)(k+3)(2+\sqrt{2})/h$. 
While the
ratio between  $\Vol_{d-1}(\partial T)$ and $\Vol_{d}(T)$ can be arbitrarily
large, $\gamma_1(k,\tau_h)$ is readily
bounded when we consider quasiuniform families of meshes
\cite[Definition 4.4.13]{MR2373954}, where $\Vol_{d-1}(\partial T)$ of each $d$-simplex
$T$ is bounded above by $\mathcal{O}(h^{d-1})$ and $\Vol_d(T)$ is bounded
below by $\mathcal{O}(h^d)$.  This naturally leads to an approximation
property for the trace norm. These results will be useful in the analysis
of the Nitsche boundary integrals. % We note that Lemma~\ref{Cor4} adds an important restriction, that $\Omega$ be polygonal or polyhedral, in order for the error estimate given there to hold.  This is required only in Theorem~\ref{cor1} below; for the remainder of the analysis in Section~\ref{sec:discrete}, we only require that $\Omega$ is a bounded, Lipschitz, and connected domain.
\begin{corollary}\label{cor:inverse_trace_quasi}
	Let $\{\tau_h\}$, $0< h \leq 1$ be a family of quasiuniform meshes of the domain $\Omega\subset \mathbb{R}^d$, $d\in\{2,3\}$.  Then, there exists $C_\Omega > 0$ such that for any $\tau_h$ in the family,
	\begin{align*}
	\forall u_h\in DG_{k}(\Omega,\tau_h), \|u_h\|_{0,\partial\Omega}^2&\leq \frac{\gamma_1(k)}{h}\|u_h\|_0^2, \\
	\forall \vec{v}_h\in RT_{k+1}(\Omega,\tau_h),  \|\vec{v}_h\cdot\vec{n}\|^2_{0,\partial\Omega}&\leq \frac{\gamma_1(k+1)}{h}\|\vec{v}_h\|_0^2,
	\end{align*}
	where 
	\begin{equation}
	\gamma_1(k)= C_\Omega(k+1)(k+d) \geq h\max_{T\in \partial\tau_h}\frac{(k+1)(k+d) \ \Vol_{d-1}(\partial T)}{d \ \Vol_{d}(T) },
	\end{equation}
        for all $\tau_h$. The constant $C_\Omega$ is determined by the dimension, $d$, and the quasiuniformity parameter for the family.
        
%	and
%	\begin{equation}
%	\gamma_2(k)= h\max_{T\in \partial\tau_h}\frac{(k+2)(k+d+1) \ \text{Vol}(\partial T)}{d \ \text{Vol}(T) } \leq C_\Omega(k+2)(k+d+1).
%	\end{equation} 
\end{corollary}
%\begin{proof}
% From the definition of a quasiuniform family of meshes, we have that
%  \begin{align*}
%    \text{Perimeter length}(T) & \leq 3\text{diam}(T) \leq 3h\text{diam}(\Omega), \\
%    \text{Area}(T) & \geq \text{Area}(B_T) = \frac{\pi}{4}\left(\text{diam}(B_T)\right)^2 \geq \frac{\pi}{4}\left(\rho h \text{diam}(\Omega)\right)^2,
%  \end{align*}
%  where $B_T$ is the incircle of $T$ and $\rho$ is the constant from the definition of quasiuniformity.
%\end{proof}
\begin{lem}\label{cor:inverse_trace_quasi_approx}
	Let $\{\tau_h\}$, $0< h \leq 1$ be a family of quasiuniform meshes of
	the domain $\Omega\subset \mathbb{R}^d$, $d\in\{2,3\}$, and $\vec{v}\in
	\left[H^{k+2}(\Omega)\right]^d$. Then, there exists a constant, $m_1$,
	such that
	\[
	\left\|\left(\vec{v} - \Pi_h^{k+1}\vec{v}\right)\cdot
	\vec{n}\right\|_{0,\partial\Omega} \leq m_1 h^{k+1/2}\revise{|\vec{v}|_{k+1}},
	\]
where $\Pi^{k+1}_h$ is the natural Raviart-Thomas interpolation operator defined in Theorem~\ref{pro}.  
\end{lem}

\begin{proof}
\revise{For any $T\in\tau_h$, we have $\vec{v}-\Pi^{k+1}_h\vec{v}\in
[H^{k+2}(T)]^d$.  By \cite[Lemma 12.15]{MR4242224}, there exists a constant $C_1$, independent of $h$,  such that the following inequality holds for every triangle, $T\in\tau_h$,}
\revise{\begin{eqnarray}\label{Trace-Ineq}
	\left\|\left(\vec{v} -
	\Pi_h^{k+1}\vec{v}\right)\cdot\vec{n}\right\|^2_{0,\partial T}\leq C_1\left(h\|\nabla\left(\vec{v}-\Pi^{k+1}_h\vec{v}\right)\|_{0,T}^2+h^{-1}\|\left(\vec{v}-\Pi^{k+1}_h\vec{v}\right)\|_{0,T}^2\right).
\end{eqnarray}
}\revise{Theorem~\ref{pro} gives 
\begin{eqnarray}\label{key2}
	h^{-1}\|\vec{v}-\Pi^{k+1}_h\vec{v}\|^2_{0,T}\leq  c_1^2h^{2k+1}|\vec{v}|^2_{k+1,T},
\end{eqnarray} 
where $c_1$ is the constant defined in Theorem~\ref{pro}. The error estimate in \cite[Proposition 2.5.1]{MR3097958} implies the existence of a constant $c_3$ such that }
\revise{\begin{eqnarray}\label{key1}
	h\|\nabla \left(\vec{v}-\Pi^{k+1}_h\vec{v}\right)\|_{0,T}^2\leq c_3h^{2k+1}|\vec{v}|^2_{k+1,T}.
\end{eqnarray} 
}
\revise{Summing \eqref{key2} and \eqref{key1} over all elements, $T$, with
$\partial T \cap \partial\Omega \neq 0$ gives
}
\revise{\begin{eqnarray*}
	\left\|\left(\vec{v} - \Pi_h^{k+1}\vec{v}\right)\cdot
	\vec{n}\right\|^2_{0,\partial\Omega}\leq 	\left\|\vec{v}-\Pi_h^{k+1}\vec{v}\right\|^2_{0,\partial\Omega}\leq C_1\left(c_1^2 + c_3\right) h^{2k+1}|\vec{v}|^2_{k+1}.
\end{eqnarray*}}
\revise{Thus, we have}
\[\revise{\left\|\left(\vec{v} - \Pi_h^{k+1}\vec{v}\right)\cdot
\vec{n}\right\|_{0,\partial\Omega} \leq m_1 h^{k+1/2}|\vec{v}|_{k+1},}
\]
\revise{where $m_1= \sqrt{C_1\left(c_1^2 + c_3\right)}$. }
\end{proof}
\begin{lem}\label{Cor4}
Let $\{\tau_{h}\}$, $0<h\leq 1$ be a family of quasiuniform meshes of $\Omega$, and $\vec{v}\in [H^{k+2}(\Omega)]^d$. Then, there exists a positive constant, $m_2$, such that 
\begin{eqnarray*}
\|\nabla\cdot(\vec{v}-\Pi^{k+1}_h\vec{v})\|_{\partial\Omega}\leq m_2 h^{k+1/2}|\vec{v}|_{k+2}.
\end{eqnarray*}
where $\Pi^{k+1}_h$ is the Raviart-Thomas interpolation operator defined in Theorem~\ref{pro}.
\end{lem}
\begin{proof}\revise{
For any $T\in\tau_h$, we have
$\nabla\cdot\left(\vec{v}-\Pi_h^{k+1}\vec{v}\right)\in H^{k+1}(T)$ and, similar to Inequality \eqref{Trace-Ineq}, we have a constant $C_2$, independent of $h$ such that the following inequality holds for every triangle $T\in\tau_h$,}
\revise{\begin{eqnarray*}
\left\|\nabla\cdot\left(\vec{v} -
\Pi_h^{k+1}\vec{v}\right)\right\|^2_{0,\partial T}\leq C_2\left(h\|\nabla\nabla\cdot \left(\vec{v}-\Pi^{k+1}_h\vec{v}\right)\|_{0,T}^2+h^{-1}\|\nabla\cdot\left(\vec{v}-\Pi^{k+1}_h\vec{v}\right)\|_{0,T}^2\right).
\end{eqnarray*}}
%Note that for $k\geq 1$, $\Pi_h^{k+1} \vec{v}\in RT_{k+1}(\Omega,\tau_h)$ and $[CG_{k}(\Omega,\tau_h)]^d\subset RT_{k+1}(\Omega,\tau_h)$. That is, 
\revise{Using Theorem~\ref{pro}, we have the error estimate }
\revise{\begin{eqnarray}\label{key4}
	h^{-1}\|\nabla\cdot\left(\vec{v}-\Pi^{k+1}_h\vec{v}\right)\|_{0,T}^2\leq c_2^2 h^{2k+1}|\vec{v}|^2_{k+2,T},
\end{eqnarray}}
\revise{where $c_2$ is defined in Theorem~\ref{pro}. By \cite[Proposition 2.5.2]{MR3097958} and \cite[Equation 2.5.15]{MR3097958}, on a triangle $T$, we have $\nabla\cdot\Pi_h^{k+1}\vec{v}=\mathscr{I}^k_h\nabla\cdot\vec{v}$, where $\mathscr{I}^k_h$ is the $L^2$-projection on $P_k(T)$. Then, \cite[Theorem 18.16]{MR4242224} gives}
\revise{\begin{eqnarray}\label{key3}
h\|\nabla\nabla\cdot \left(\vec{v}-\Pi^{k+1}_h\vec{v}\right)\|_{0,T}^2=h\|\nabla\left(\nabla\cdot\vec{v}-\mathscr{I}_h^{k}\nabla\cdot\vec{v}\right)\|_{0,T}^2\leq c_4 h^{2k+1}|\vec{v}|^2_{k+2,T},
\end{eqnarray}}
\revise{where $c_4$ is a positive constant independent of $h$.  Summing over all elements, $T$, with $\partial T\cap \partial\Omega\neq \emptyset$ and taking the square root gives }
\revise{	\begin{eqnarray*}
		\|\nabla\cdot(\vec{v}-\Pi^{k+1}_h\vec{v})\|_{\partial\Omega}\leq m_2 h^{k+1/2}|\vec{v}|_{k+2},
	\end{eqnarray*}}
\revise{where $m_2= \sqrt{C_2\left(c_2^2 + c_4\right)}$. }
\end{proof}

\section{Continuum Analysis}
\label{sec:continuum}

Consider the fourth-order problem \eqref{eq:mod_biharmonic} with suitable boundary conditions (discussed below),
\begin{equation}\label{eq2}
\Delta^2 u- c_0\Delta u+c_1u = f \quad \text{ in } \Omega,
\end{equation}
where $\Omega\subset \mathbb{R}^d$, $d\in\{2,3\}$ is a bounded, Lipschitz, and connected domain with outward pointing normal $\vec{n}$, and $c_0$ and $c_1$ are nonnegative constants. Define $V=\{v\in H^1(\Omega)\ |\ \Delta v \in L^2(\Omega) \}$ with dual space $V^*$, and assume that $f\in V^*$.  Multiplying by a test function, $v\in V$, integration by parts yields
%\pef{We've now additionally assumed that $\Omega$ is polytopal. If this is necessary, assume it earlier, at the start of section 2. If it is not, remove it.}
\begin{equation}\label{eq:int_by_parts}
\int_\Omega \Delta u\Delta v + c_0\nabla u \cdot \nabla v + c_1 uv + \int_{\partial\Omega}v(\nabla\Delta u-c_0\nabla u)\cdot\vec{n}-\int_{\partial\Omega}\Delta u \nabla{v}\cdot\vec{n} = \int_\Omega fv.
\end{equation}
% Then given any $v\in V$, it follows from integration by parts that 
%\begin{eqnarray}\label{wellposed}
%\int_{\Omega}fv&=&a(u,v)+ \int_{\partial\Omega}v(\nabla\Delta u-c_0\nabla u)\cdot\vec{n}-\int_{\partial\Omega}\Delta u \nabla{v}\cdot\vec{n}, 
%\end{eqnarray}

Since \eqref{eq2} is a fourth-order problem, we require two boundary conditions on any segment of $\partial \Omega$. Here, we focus on the boundary conditions that arise naturally from \eqref{eq:int_by_parts}:
%Different boundary conditions can be considered on segments of $\partial \Omega$, with each segment requiring two BCs, in the pairs given below which are naturally arise from the weak form.  Consequently, we write $\partial \Omega=\Gamma_0\cup\Gamma_1\cup\Gamma_2\cup\Gamma_3$ with $\Gamma_i\cap\Gamma_j = \varnothing$ for $i\neq j$, and specify
\begin{alignat}{3}
u&=0,\quad &\Delta u&=0,\quad &&\text{on}\ \Gamma_0,\label{55}\\    
u&=0, &\frac{\partial u}{\partial n}&=0,&& \text{on}\ \Gamma_1, \ \label{!!}\\
\frac{\partial (\Delta u-c_0u)}{\partial n}&=0, & \Delta u&=0, && \text{on}\ \Gamma_2,\label{!!!}\\
\frac{\partial (\Delta u-c_0u)}{\partial n}&=0, &\frac{\partial u}{\partial n}&=0, && \text{on}\ \Gamma_3\label{!!!!}.
\end{alignat}	
Note that we consider homogeneous boundary conditions here,  but the results hold true for nonhomogeneous boundary conditions if the traces of these quantities are smooth enough on $\partial\Omega$, using standard techniques (cf.~\cite{MR851383}) to transform the inhomogeneous boundary conditions \eqref{55}-\eqref{!!!!} into homogeneous ones. We note that the commonly-considered case of clamped boundary conditions corresponds to $\Gamma_1$ in this classification.
Under suitable assumptions on $c_0$ and $c_1$, we can prove a variety
of results on the well-posedness of the resulting variational problems in the
standard Hilbert space setting.
\begin{lem}\label{lem:hilbert}
	Equip $V$  with the inner product
	\begin{equation*}
	(u,v)_V=\int_{\Omega}uv+\nabla u\cdot\nabla v+\Delta u\Delta v.
	\end{equation*}
	The normed space $\left(V,\|.\|_V\right)$ is a Hilbert space. 
\end{lem}

Defining $V_0=\{v\in V\ | \  v=0 \ \text{on}\ \Gamma_0\cup\Gamma_1\ \text{and} \ \frac{\partial v}{\partial n} =0 \ \text{on} \ \Gamma_1\cup\Gamma_3 \}$, we first consider the $H^2$-conforming weak form of \eqref{eq2}, requiring $u\in V_0$ such that 
\begin{equation}\label{wellposed}
a(u,v) = \int_{\Omega}fv, \quad\forall v\in V_0,
\end{equation}
where the bilinear form $a$ is defined as
\begin{equation}
a(u,v)=\int_{\Omega}\Delta u\Delta v +c_0\nabla u\cdot\nabla v+c_1uv.
\end{equation} 

Using standard tools, it is straightforward to show that the weak form in \eqref{wellposed} is well-posed (i.e., that $a(u,v)$ is coercive and continuous on $V_0$) when $c_0,c_1 > 0$, for any choice of boundary conditions. This can be extended to cover the case of $c_0 = 0$ if $\Gamma_2 = \emptyset$, using the remaining boundary conditions to show that there is a constant, $C$, such that $\|\nabla u\|_0^2 \leq C\left(\|u\|_0^2 + \|\Delta u\|_0^2\right)$.  If $c_1 = 0$ for $c_0> 0$, then well-posedness requires that $\Gamma_0 \cup \Gamma_1 \neq \emptyset$, in order to be able to apply the standard Poincar\'e inequality.  If both $c_0 = c_1 = 0$, then both $\Gamma_0 \cup \Gamma_1 \neq \emptyset$ and $\Gamma_2 = \emptyset$ are required to show well-posedness. 

\begin{rmk}
	
	When $\partial\Omega=\Gamma_2$ (the analogous case to full Neumann boundary conditions), if $c_0=0$, then $a(u,v)$ is not coercive on $V_0 = V$. We illustrate this by considering the harmonic function $v=e^{-kx}\cos(ky)$, for which $a(v,v)=c_1\|v\|_0^2=\mathcal{O}({k^{-2}})$. On the other hand, $\|v\|^2_V=\mathcal{O}({k^{-2}})+\mathcal{O}(1)$. Thus, as $k$ gets larger, the implied bound on the coercivity constant goes to zero. Thus, in what follows, $c_0$ is restricted to be positive if $\Gamma_2\subseteq \partial\Omega$.
	
\end{rmk}

We now turn our attention to the mixed formulation at the continuum level. Letting $\vec{v}=\nabla u \ \text{and} \ \vec{\alpha}=\nabla \nabla \cdot\vec{v}-c_0\vec{v}$, \eqref{eq2} is equivalent to the following system of first- and second-order PDEs. 
\begin{eqnarray}
\nabla\cdot \vec{\alpha}+c_1u=f, \label{eq:div_alpha}\\
\vec{\alpha}-\nabla \nabla\cdot \vec{v}+c_0\vec{v}=0, \label{eq:graddiv_v}\\
\vec{v}-\nabla u=0.\label{eq:grad_u}
\end{eqnarray}
Considering the relevant spaces and applying the boundary conditions given in \eqref{55}-\eqref{!!!!}, the weak form of \eqref{eq:div_alpha}-\eqref{eq:grad_u} is to find the triple 
$(u,\vec{v}, \vec{\alpha})\in L^2(\Omega)\times H^{\Gamma_1\cup\Gamma_3}_{0}(\text{div};\Omega)\times H_{0}^{\Gamma_2\cup\Gamma_3}(\text{div};\Omega)$ such that 
\begin{eqnarray}
\int_{\Omega}^{}\nabla \cdot \vec{\alpha } \phi+ c_1u\phi&=&\int_{\Omega}^{ }f\phi,\\
\int_{\Omega}^{}\vec{\alpha }\cdot\vec{\psi}+\nabla\cdot\vec{v}\,\nabla\cdot\vec{\psi}+c_0\int_{\Omega}\vec{v}\cdot\vec{\psi}&=&0, \label{eq:weak_psi}\\
\int_{\Omega}^{}\vec{\beta}\cdot\vec{v}+\int_{\Omega}u\nabla\cdot\vec{\beta}&=&0,\label{3.16}
\end{eqnarray}
$\forall(\phi, \vec{\psi}, \vec{\beta}) \in L^2(\Omega)\times H_{0}^{\Gamma_1\cup\Gamma_3}(\text{div};\Omega) \times H_{0}^{\Gamma_2\cup\Gamma_3}(\text{div};\Omega)$,
where, for $\Gamma \subset \partial\Omega$,
\begin{equation*}
H_{0}^{\Gamma}(\text{div};\Omega) =\left\{\vec{v}\in H(\text{div};\Omega)\ | \ \vec{v}\cdot\vec{n}=0\ \text{on} \ \Gamma \right\}.
\end{equation*}
We note that this formulation strongly imposes Dirichlet boundary conditions on $\vec{v}$ and $\alpha$, but weakly imposes those on $u$ and $\Delta u$.

This weak form is equivalent to the saddle-point problem of finding $(u,\vec{v}, \vec{\alpha})\in L^2(\Omega)\times H_{0}^{\Gamma_1\cup\Gamma_3}(\text{div};\Omega)\times H_{0}^{\Gamma_2\cup\Gamma_3}(\text{div};\Omega)$ such that 
\begin{eqnarray}
a\big((u,\vec{v}), (\phi, \vec{\psi})\big)+ {b}\big((\phi, \vec{\psi}),\vec{\alpha} \big)&=&F(\phi),\label{454}\\
{b}((u,\vec{v} ), \vec{\beta}) &=&0,\label{444}
\end{eqnarray}
$\forall(\phi, \vec{\psi}, \vec{\beta}) \in L^2(\Omega)\times H_{0}^{\Gamma_1\cup\Gamma_3}(\text{div};\Omega) \times H_{0}^{\Gamma_2\cup\Gamma_3}(\text{div};\Omega)$, where the linear and bilinear forms $a, \ b, \ \text{and} \ F$ are given by
\begin{eqnarray}
a\big((u,\vec{v}), (\phi, \vec{\psi})\big)&=&c_0\int_{\Omega}\vec{v}\cdot\vec{\psi} +\int_{\Omega}\nabla\cdot\vec{v}\,\nabla\cdot\vec{\psi}+c_1\int_{\Omega}u\phi,\label{28}\\
b\big((u,\vec{v}), \vec{\beta}\big)&=&\int_{\Omega}\vec{\beta}\cdot\vec{v}+u\nabla\cdot \vec{\beta},\label{29}\\
F(\phi)&=&\int_{\Omega}f\phi.\label{30}
\end{eqnarray}
%The right hand sides, $F$ and $G$, depend on the boundary conditions we impose on any segment of $\partial\Omega$.
As noted above, the boundary conditions imposed can have significant effects on the well-posedness of the problem. In particular, we now show that the mixed-formulation is well-posed under combinations of assumptions on $c_0, \ c_1$, and the boundary conditions. 
\begin{thm}\label{thm6}
	Let $\partial\Omega=\Gamma_0\cup\Gamma_3$. The saddle-point problem \eqref{454}-\eqref{444} has a unique solution for any  $c_0\geq 0 \ \text{and}\ c_1>0$, and for $c_1 \ge 0$ if $\Gamma_0$ is nonempty.
\end{thm}
\begin{proof}
	We verify the standard Brezzi conditions for well-posedness \cite{MR3097958}. Continuity of $a$, $b$, and $F$ in the product norm on $L^2(\Omega)\times H(\text{div};\Omega)\times H(\text{div};\Omega)$ is straightforward.
	
	We next show that the bilinear form $a\big((u,\vec{v}), (\phi, \vec{\psi})\big)$ is coercive on the set 
	\begin{equation*}
	\eta=\{(u,\vec{v})\in L^2(\Omega)\times H_0^{\Gamma_3}(\text{div};\Omega)\ | \  b\big( (u,\vec{v}),\vec{\alpha}\big)=0, \ \ \forall \vec{\alpha}\in H_0^{\Gamma_3}(\text{div};\Omega)\}.
	\end{equation*}  
	\revise{Since $\vec{v}$ is also in $H_0^{\Gamma_3}(\text{div};\Omega)$,} the kernel condition implies that $b\big((u,\vec{v}),\vec{v}\big)=0$ for any $(u,\vec{v})$ in $\eta$, which implies that $\|\vec{v}\|_0^2=-\int_{\Omega}^{ }u\nabla\cdot\vec{v}\leq \frac{1}{2}\left(\|u\|_0^2+\|\nabla\cdot\vec{v}\|_0^2\right)$. Then
	\begin{align*} 
		a\big((u,\vec{v}),(u,\vec{v})\big)	&=	c_0\|\vec{v}\|_0^2+\frac{1}{3}\big(\|\nabla\cdot\vec{v}\|_0^2+c_1\|u\|_0^2\big)+\frac{2}{3}\big(\|\nabla\cdot\vec{v}\|_0^2+c_1\|u\|_0^2\big)\notag\\
		&\geq c_0\|\vec{v}\|_0^2+\frac{2\min{\{1,c_1\}}}{3}\|\vec{v}\|_0^2+\frac{2\min{\{1,c_1\}}}{3}\big(\|\nabla\cdot\vec{v}\|_0^2+\|u\|_0^2\big)\notag\\
		%&\geq&	%c_0\|\vec{v}\|_0^2+\frac{-2\min\{1,c_1\}}{3}\int_{\Omega}^{ }u\nabla\cdot\vec{v}+\frac{2\min\{1,c_1\}}{3}\big(\|\nabla\cdot\vec{v}\|_0^2+\|u\|_0^2\big)\notag\\
		%	&=&\left(c_0+\frac{2\min\{1,c_1\}}{3}\right)\|\vec{v}\|_0^2+\frac{2\min\{1,c_1\}}{3}\big(\|\nabla\cdot\vec{v}\|_0^2+\|u\|_0^2\big)\\
		&\geq  \frac{2\min\{1,c_1\}}{3}\left(\|\vec{v}\|^2_{\text{div}}+\|u\|_0^2\right),
	\end{align*}
	where $\|\vec{v}\|^2_{\text{div}}=\|\vec{v}\|^2_0+\|\nabla\cdot\vec{v}\|_0^2$.
	
	If $\Gamma_0$ is nonempty and $c_1=0$, then for a given $(u,\vec{v})$, we choose $\vec{\alpha}=\mu\vec{v}+\vec{\alpha}_m$, where $\mu$ is a positive constant to be specified below, \revise{and $\vec{\alpha}_m\in H^{\Gamma_3}_0(\text{div};\Omega)$ is the solution} of the standard mixed Poisson problem, 
	\begin{eqnarray}
	\int_{\Omega}\delta\nabla\cdot\vec{\alpha}_m&=&\int_{\Omega}u\delta, \ \ \forall \delta \in L^2(\Omega),\label{eq2.44}\\
	\int_{\Omega}\vec{\alpha}_m\cdot\vec{\beta}+\phi\nabla \cdot\vec{\beta}&=&0, \ \ \forall \vec{\beta} \in H_0^{\Gamma_3}(\Omega;\text{div}),\label{eq2.45}
	\end{eqnarray}
	which is well-posed with $\|\vec{\alpha}_m\|^2_{\text{div}}+\|\phi\|_0^2\leq \Lambda\|u\|_0^2$, where $\Lambda$ is a positive constant that depends on the coercivity and continuity constants and the inf-sup conditions for the mixed Poisson problem~\cite{10.1093/climsys/dzw005,MR3097958}. Moreover, the choice of $\delta=u$ in Equation \eqref{eq2.44} implies that $\|u\|_0^2=\int_{\Omega}u\nabla\cdot\vec{\alpha}_m$. Thus, for every $(u,\vec{v})\in \eta$, we have     
	\[
	b((u,\vec{v}),\vec{\alpha}) = \mu\|\vec{v}\|_0^2+\|u\|_0^2+\int_{\Omega}\vec{\alpha}_m\cdot\vec{v}+\mu\int_{\Omega}u\nabla\cdot\vec{v}=0.
	\]
	Rearranging terms and using the Cauchy-Schwarz and Young's inequalities, we get
	\[
	\frac{\mu}{2}\left(\frac{2\mu}{k_1} \|u\|_0^2+\frac{k_1}{2\mu}\|\nabla\cdot\vec{v}\|_0^2\right)+\frac{1}{2}\left(\frac{2}{k_2}\|u\|_0^2+\frac{k_2\Lambda}{2}\|\vec{v}\|_0^2\right)\geq\mu\|\vec{v}\|_0^2+\|u\|^2_0
	\]
	for arbitrary $k_1 > 0, k_2 > 0$, which can be further rearranged to yield
	\[
	\frac{k_1}{4}\|\nabla\cdot\vec{v}\|_0^2\geq\left(\mu-\frac{k_2\Lambda}{4}\right)\|\vec{v}\|_0^2+\left(1-\frac{\mu^2}{k_1}-\frac{1}{k_2}\right)\|u\|_0^2.
	\]
	\revise{Choosing the constants $k_1=4\mu^2$, $k_2=4$ and $\mu=\Lambda+\frac{1}{2}$ results in the coercivity condition that
	$a\left((u,\vec{v}),(u,\vec{v})\right)=\|\nabla\cdot\vec{v}\|_0^2\geq K \left(\|\vec{v}\|_{\text{div}}^2+\|u\|_0^2\right)$, where $K=\frac{1}{2\left(\Lambda+\frac{1}{2}\right)^2}$.}
	
	Finally, we establish the necessary inf-sup condition, that
	\begin{equation*}
	\sup_{\substack{(u,\vec{v})\in L^2(\Omega)\times H_0^{\Gamma_3}(\text{div};\Omega)\\ (u,\vec{v})\neq (0,\vec{0})}} \frac{b\big((u,\vec{v}),\vec{\alpha}\big)}{\sqrt{\|u\|_0^2+\|\vec{v}\|^2_{{\text{div}}}}}\geq \frac{1}{\sqrt{2}} \|\vec{\alpha}\|_{{\text{div}}}, \ \forall \vec{\alpha }\in H_0^{\Gamma_3}(\text{div};\Omega) 
	\end{equation*}
	The choice $u=\nabla\cdot\vec{\alpha}, \ \vec{v}=\vec{\alpha}$ completes the proof, noting this is compatible with $\partial\Omega=\Gamma_0\cup\Gamma_3$, since $u\in L^2(\Omega)$, without an essential boundary condition strongly imposed on it.
\end{proof}

\begin{corollary}\label{corr}
	Let $\partial\Omega=\Gamma_0\cup\Gamma_2\cup\Gamma_3$. The saddle-point problem \eqref{454}-\eqref{444} has a unique solution for any  $c_0> 0 \ \text{and}\ c_1>0$.
\end{corollary}
\begin{proof}
	Under these assumptions, the bilinear form $a$ is coercive for $(u,\vec{v})\in L^2(\Omega)\times H_0^{\Gamma_3}(\text{div};\Omega)$ since  
	\begin{equation*}
	a\left((u,\vec{v}), (u,\vec{v})\right)=c_0\|\vec{v}\|_0^2+\|\nabla\cdot\vec{v}\|_0^2+c_1\|u\|_0^2 \ge \min{\{1,c_0,c_1\}}\left(\|\vec{v}\|_{\text{div}}^2 + \|u\|_0^2\right).
	\end{equation*}
 Moreover, the inf-sup condition,
	\begin{equation*}
	\sup_{\substack{(u,\vec{v})\in L^2(\Omega)\times H_0^{\Gamma_3}(\text{div};\Omega)\\ (u,\vec{v})\neq (0,\vec{0})}} \frac{b\big((u,\vec{v}),\vec{\alpha}\big)}{\sqrt{\|u\|_0^2+\|\vec{v}\|^2_{{\text{div}}}}}\geq \frac{1}{\sqrt{2}} \|\vec{\alpha}\|_{{\text{div}}}, \ \forall \vec{\alpha }\in H_0^{\Gamma_2\cup\Gamma_3}(\text{div};\Omega) ,
	\end{equation*}
	is readily shown by choosing $u=\nabla\cdot\vec{\alpha}, \ \vec{v}=\vec{\alpha}$, noting that this is allowable because $\vec{\alpha}\in H_0^{\Gamma_2\cup\Gamma_3}(\text{div};\Omega)\subset H_0^{\Gamma_3}(\text{div};\Omega)$, and $\nabla\cdot\vec{\alpha}\in L^2(\Omega)$.
\end{proof}
Solving \eqref{454}--\eqref{444} when essential boundary conditions on $\vec{v}$ are strongly imposed while $\vec{\alpha}$ is free on the boundary, i.e.~$\partial\Omega = \Gamma_1$, leads to difficulties in proving the inf-sup condition. This difficulty can easily be understood from the proof of the inf-sup condition in Theorem \ref{thm6}, in which we take $\vec{v}=\vec{\alpha}$ to provide a concrete bound on the supremum.  In this setting, we are able to prove uniqueness of the solution to the continuum mixed form of the problem under suitable regularity assumptions.

\begin{corollary}\label{co}
	Let $\partial\Omega=\Gamma_0\cup\Gamma_1\cup\Gamma_3$ \revise{and $\bar{u}$ solve \eqref{wellposed}}. The pair $(u,\vec{v})$ that solves the saddle-point problem \eqref{454}-\eqref{444} is unique for any  $c_0\geq 0 \ \text{and}\ c_1\geq0$ with \revise{$(u,\vec{v})=\left(\bar{u},\nabla \bar{u}\right)$. Moreover, $\vec{\alpha}$ is unique if $\bar{u}\in H^p(\Omega)$, ${p}\geq 4$.}
\end{corollary}
\begin{proof}
	
	As in the proof of Theorem \ref{thm6}, the bilinear form $a$ is coercive on the set 
	\begin{equation*}
	\eta=\{(u,\vec{v})\in L^2(\Omega)\times H_0^{\Gamma_1\cup\Gamma_3}(\text{div};\Omega)\ | \  b\big( (u,\vec{v}),\vec{\alpha}\big)=0, \ \ \forall \vec{\alpha}\in H_0^{\Gamma_3}(\text{div};\Omega)\}.
	\end{equation*}  
        Therefore, the pair $(u,\vec{v})$ is uniquely determined~\cite[Remark 1.1]{MR1115205}. \revise{As $(\bar{u},\nabla u)$ solves \eqref{444} for every $\vec{\beta}$, uniqueness of $(u,\vec{v})$ implies that $(u, \vec{v})= (\bar{u}, \nabla\bar{u})$.}
        If, additionally, $u=\bar{u}\in H^{p}(\Omega)$, ${p}\geq 4$, then we choose $(\phi,\vec{\psi})=\left(0,Q\left(\vec{\alpha}-\nabla\Delta u+c_0\nabla u\right)\right)$ in \eqref{454}, for any \revise{$Q\in C^{{p}-3}(\Omega)\cap H^1_0(\Omega)$} that is positive in $\Omega$. Note that $\vec{\psi} \in H_0^{\Gamma_1\cup\Gamma_3}(\text{div};\Omega)$ since $u\in H^p(\Omega)$ for $p\geq 4$. Integration by parts on \eqref{454} then yields
	\begin{eqnarray*}
		\int_{\Omega}Q\left(\vec{\alpha}-\nabla\Delta u+c_0\nabla u\right)\cdot \left(\vec{\alpha}-\nabla\Delta u+c_0\nabla u\right)=0.
	\end{eqnarray*}   
	As $Q\left(\vec{\alpha}-\nabla\Delta u+c_0\nabla u\right)\cdot \left(\vec{\alpha}-\nabla\Delta u+c_0\nabla u\right)$ is non-negative in $\Omega$, this implies that $\vec{\alpha}=\nabla\Delta u-c_0\nabla u$.  
\end{proof}
\begin{rmk}
  \revise{In the case $d=2$, we can write $\partial\Omega=\displaystyle\left\{\cup_{i=1}^{M_1}\Gamma^i\right\} \cup \left\{\cup_{i=1}^{M_2} \tilde{\Gamma}^i\right\}$, where $\Gamma^{i}=(x,a_ix+b_i)$ for $x_i^0 < x < x_i^1$, and $\tilde{\Gamma}^{i}=(c_iy+d_i,y)$  for $y_i^0 < y < y_i^1$, with $M_1$ and $M_2$ positive integers. The function $Q$ in Corollary \eqref{co} can be chosen as
    \[
    Q= \Pi_{i=1}^{M_1}\left(y-a_ix-b_i\right)^2\Pi_{i=1}^{M_2}\left(x-c_iy-d_i\right)^2,
    \]
    with $Q\in C^{{p}-3}(\Omega)\cap H^1_{0}(\Omega)$ and $Q$ positive in the interior of $\Omega$. Similarly, $Q$ can be constructed when $d=3$ by writing the boundary faces of $\Omega$ in sets that can be parametrized as planes in each pair of two Cartesian coordinates.  
}\end{rmk}
\section{Discrete Analysis}
\label{sec:discrete}

%We first note that, as in the continuum, if $\partial\Omega = \Gamma_0\cup\Gamma_3$, then any conforming discretization is well-posed.
For what follows, we consider a conforming discretization of the mixed form, with
\[
(u_h,\vec{v}_h,\vec{\alpha}_h) \in DG_k(\Omega,\tau_h)\times RT^{\Gamma_3}_{k+1}(\Omega,\tau_h)\times RT^{\Gamma_2\cup\Gamma_3}_{k+1}(\Omega,\tau_h),
\]
for $k\geq 0$, where $RT^{\Gamma}_{k+1}(\Omega,\tau_h)=\left\{\vec{v}_h\in RT_{k+1}(\Omega,\tau_h)\ | \ \vec{v}_h\cdot\vec{n}=0\ \text{on} \ \Gamma\right\}$, noting that $DG_k(\Omega,\tau_h) \subset L^2(\Omega)$ and $RT^{\Gamma}_{k+1}(\Omega,\tau_h) \subset H_0^{\Gamma}(\text{div};\Omega)$. We examine the problem of finding $(u,\vec{v}, \vec{\alpha})\in DG_{k}(\Omega,\tau_h)\times RT_{k+1}^{\Gamma_3}(\Omega,\tau_h)\times RT^{\Gamma_2\cup\Gamma_3}_{k+1}(\Omega,\tau_h)$ such that
\begin{eqnarray}
{a}\big((u_h,\vec{v}_h ), (\phi_h, \vec{\psi}_h)\big)+ {b}\big((\phi_h, \vec{\psi}_h),\vec{\alpha}_h \big)&=&{F}\big(\phi_h),\label{eq:conf-disc-aa}\\
{b}\big((u_h,\vec{v}_h), \vec{\beta}_h\big)&=&0,\label{eq:conf-disc-ba}
\end{eqnarray}
$\forall (\phi_h,\vec{\psi}_h,\vec{\beta}_h)\in DG_{k}(\Omega,\tau_h)\times RT^{\Gamma_3}_{k+1}(\Omega,\tau_h)\times RT^{\Gamma_2\cup\Gamma_3}_{k+1}(\Omega,\tau_h)$, where $a$, $b$ and $F$ are defined in \eqref{28}-\eqref{30}. Boundary conditions on $\Gamma_1$ will be enforced with Nitsche's method. As in the continuum case, taking $\partial\Omega=\Gamma_0\cup\Gamma_3$ is the easiest case to consider.
\begin{corollary}\label{cor:discrete_gamma0}
	Let $\partial\Omega = \Gamma_0\cup \Gamma_3$, $c_0\geq0, c_1>0$,  and $c_1\geq0$ if $\Gamma_0$ is nonempty. Let $\{\tau_h\}$, $0 < h \leq 1$ be a quasiuniform family of triangular meshes of $\Omega$.  Then problem \eqref{eq:conf-disc-aa}-\eqref{eq:conf-disc-ba} has a unique solution for any $\tau_h$ in the family.
\end{corollary}
\begin{proof}
	Coercivity of the bilinear form $a\big((u_h,\vec{v}_h), (\phi, \vec{\psi}_h)\big)$ on the set 
	\begin{equation*}
	\{(u_h,\vec{v}_h)\in DG_{k}(\Omega,\tau_h)\times RT^{\Gamma_3}_{k+1}(\Omega,\tau_h)\ | \  b\big( (u_h,\vec{v}_h),\vec{\alpha}_h\big)=0, \ \ \forall \vec{\alpha}\in RT^{\Gamma_3}_{k+1}(\Omega,\tau_h)\}, 
	\end{equation*}  
	and the inf-sup condition of the form
	\begin{equation*}
	\sup_{\substack{(u_h,\vec{v}_h)\in DG_{k}(\Omega,\tau_h)\times  RT^{\Gamma_3}_{k+1}(\Omega,\tau_h)\\(u_h,\vec{v}_h)\neq (0,\vec{0})}} \frac{b\big((u_h,\vec{v}_h),\vec{\alpha}_h\big)}{\sqrt{\|u_h\|_0^2+\|\vec{v}_h\|^2_{{\text{div}}}}}\geq \frac{1}{\sqrt{2}} \|\vec{\alpha}_h\|_{{\text{div}}}, \ \forall \vec{\alpha }_h\in  RT^{\Gamma_3}_{k+1}(\Omega,\tau_h),
	\end{equation*}
	can be proven exactly as in the continuum. Note that this is compatible with the finite-element spaces. For example, when $c_1>0$, we can choose $\vec{\alpha}_h=\vec{v}_h\in RT^{\Gamma_3}_{k+1}(\tau_h)$ in the kernel condition within the coercivity proof, and $\vec{v}_h=\vec{\alpha}_{h}\in RT^{\Gamma_3}_{k+1}(\tau_h)$ and $u_h=\nabla\cdot\vec{\alpha}_h$ in the proof of the inf-sup condition. Such a $u_h$ is in $DG_k(\Omega,\tau_h)$ by Remark \ref{1234}.
\end{proof}

\begin{corollary}\label{cor4}
	Let the assumptions of Corollary \ref{cor:discrete_gamma0} hold, and let $$(u,\vec{v},\vec{\alpha}) \in  H^{k+1}(\Omega)\times [H^{k+2}(\Omega)]^d\times [H^{k+2}(\Omega)]^d$$ be the solution of  \eqref{454}-\eqref{444}. Let $$(u_h,\vec{v}_h, \vec{\alpha}_h)\in DG_{k}(\Omega,\tau_h)\times RT^{\Gamma_3}_{k+1}(\Omega,\tau_h)\times RT^{\Gamma_3}_{k+1}(\Omega,\tau_h)$$ be the solution of \eqref{eq:conf-disc-aa}-\eqref{eq:conf-disc-ba}.  Then, there exist constants $m_1$ and $m_2$ such that
	\begin{align}
		\|(u,\vec{v}) - (u_h,\vec{v}_h)\|_{0,\text{\normalfont div}} & \leq m_1 h^{k+1}\bigg(|u|_{k+1}^2 + |\vec{v} |_{k+1}^2+|\vec{v}|_{k+2}^2 + |\vec{\alpha} |_{k+1}^2+|\vec{\alpha}|_{k+2}^2\bigg)^{1/2}, \label{eq:conforming_uv}\\
		\| \vec{\alpha} - \vec{\alpha}_h  \|_{\textrm{\normalfont div}} & \leq m_2 h^{k+1}\bigg(|u|_{k+1}^2 + |\vec{v} |_{k+1}^2+| \vec{v}|_{k+2}^2 + |\vec{\alpha} |_{k+1}^2+| \vec{\alpha} |_{k+2}^2\bigg)^{1/2}.\label{eq:conforming_alpha}
	\end{align}
\end{corollary}
\begin{proof}
	Because this is a conforming discretization, standard approximation theory for mixed finite elements (e.g.~\cite{MR3097958}) yields optimal approximation results in the product norm,
	$\|\left(u_h,\vec{v}_h\right)\|_{0,\text{div}}^2 =
	\|u_h\|_0^2 + \|\vec{v}_h\|_{\text{div}}^2$.
\end{proof}
%\begin{remark}
%Corollary \ref{cor4} and the triangle inequality can be used to show that $\vec{v}_h$ and $\vec{\alpha}_h$ are approximations of $\nabla u$ and $\nabla\Delta u-c_0\nabla u$, respectively, with $\mathcal{O}(h^{k+1})$ error estimates.
%\end{remark}
\begin{corollary}\label{thm4}
	Problem \eqref{eq:conf-disc-aa}-\eqref{eq:conf-disc-ba}, with $\partial \Omega=\Gamma_0\cup\Gamma_2\cup\Gamma_3$, has a unique solution for $c_0,c_1>0$.  If, further,  $$(u,\vec{v},\vec{\alpha}) \in  H^{k+1}(\Omega)\times [H^{k+2}(\Omega)]^d\times [H^{k+2}(\Omega)]^d$$ is the solution of  \eqref{454}-\eqref{444} and  $$(u_h,\vec{v}_h, \vec{\alpha}_h)\in DG_{k}(\Omega,\tau_h)\times RT^{\Gamma_3}_{k+1}(\Omega,\tau_h)\times RT^{\Gamma_3}_{k+1}(\Omega,\tau_h)$$ is the solution of \eqref{eq:conf-disc-aa}-\eqref{eq:conf-disc-ba}, then the error bounds in \eqref{eq:conforming_uv}--\eqref{eq:conforming_alpha} also hold for this case.
	%\pef{A small thing: if $g_2, g_3$ are not polynomials, these equalities can never hold. On the other hand, I don't want to get in to the technicalities of interpolating boundary data \dots} Moreover, this mixed finite-element formulation leads to optimal finite-element convergence as in Corollary \ref{cor4}. 
\end{corollary}       
\begin{proof}
	The proof follows exactly as those of Corollaries \ref{corr} and~\ref{cor4}. The bilinear form $a$ is coercive for every pair $(u_h,\vec{v}_h)\in DG_k(\Omega,\tau_h)\times RT^{\Gamma_3}_{k+1}(\Omega,\tau_h)$. The finite-element approximation spaces allow the choice $u_h=\nabla\cdot\vec{\alpha}_h$ and $\vec{v}_h=\vec{\alpha}_{h}$ for the inf-sup condition.  As this is a conforming discretization, standard theory yields optimal approximation results. 
\end{proof}
%Solving Equations \eqref{eq:conf-disc-a} and \eqref{eq:conf-disc-b} when essential boundary conditions on either $\vec{v}$ or $\vec{\alpha}$ are strongly imposed lead to difficulties in proving the inf-sup constant or the coercivity. This difficulty can easily be understood from the proofs of coercivity and of the inf-sup condition, in which we take $\vec{v}=\vec{\alpha}$ or the converse. These assumptions can not both be used when essential boundary conditions are applied to either  $\vec{v}$ or $\vec{\alpha}$. As an example, when $\partial \Omega=\Gamma_2$, the proof of coercivity isn't generally valid since $\vec{\alpha}$ must satisfy the prescribed BC while $\vec{v}$ is free on the boundary.  
As in the continuum case, solving \eqref{eq:conf-disc-aa}--\eqref{eq:conf-disc-ba} when essential boundary conditions on $\vec{v}$ are strongly imposed while $\vec{\alpha}$ is free on the boundary leads to difficulties in proving the inf-sup condition. When $\partial \Omega=\Gamma_1$, we cannot follow the proof technique used in Theorem \ref{thm6} and Corollary \ref{corr}, since $\vec{v}$ must satisfy the prescribed BC while $\vec{\alpha}$ is free on the boundary. In this case, the inf-sup condition has the form of finding $\tilde{c}>0$ such that
\begin{align*}
I=\sup_{\substack{(u_h,\vec{v}_h)\neq\in DG_k(\Omega,\tau_h)\times RT^{\Gamma_1}_{k+1}(\Omega,\tau_h)\\(u_h,\vec{v}_h)\neq (0,\vec{0})}} &\frac{\int_{\Omega}\vec{\alpha}_h\cdot\vec{v}_h+\int_{\Omega}u_h\nabla\cdot\vec{\alpha}_h}{\sqrt{\|u_h\|_0^2+\|\vec{v}_h\|^2_{{\text{div}}}}}\\ &>\tilde{c}\|\vec{\alpha}_h\|_{\text{div}}, \ \forall \vec{\alpha}_h\in RT_{k+1}(\Omega,\tau_h). 
\end{align*}

To understand the challenge, we consider a simply-connected domain $\Omega$. Then the two-dimensional discrete Helmholtz decomposition from Lemma \ref{lemma,helm},
\begin{equation*}
RT_{k+1}(\Omega,\tau_h)=\bigg(\nabla\times CG_{k+1}(\Omega,\tau_h)\bigg)\oplus \bigg(\grad_h DG_{k}(\Omega,\tau_h)\bigg).
\end{equation*}
For any $\vec{\alpha}_h =\grad_h z$ where $z\in DG_{k}(\Omega,\tau_h)$, then the choice $\vec{v}_h=0$ and $u_h=\nabla \cdot \vec{\alpha}_h-z$ satisfies the inf-sup condition, as
\begin{align*}
I&\geq\sup_{\substack{u_h\in DG_{k}(\Omega,\tau_h)\\ u_h\neq 0}} \frac{\int_{\Omega}u_h\nabla\cdot \vec{\alpha}_h}{\|u_h\|_0}\geq\frac{\|\nabla\cdot\vec{\alpha}_h\|_0^2-\int_{\Omega}z \nabla\cdot (\grad_h z)}{\|\nabla\cdot\vec{\alpha}_h\|_0+\|z\|_0} \\ &\geq \frac{\|\nabla\cdot\vec{\alpha}_h\|^2_0+\|\grad_h z\|_0^2}{\|\nabla\cdot\vec{\alpha}_h\|_0+c \|\grad_h z\|_0}
\geq \tilde{c}\left(\|\nabla\cdot \vec{\alpha}_h\|_0+ \|\vec{\alpha}_h\|_0\right),
\end{align*}
where the discrete Poincar\'e inequality \cite{MR2974162}, $\|z\|_0\leq c\|\grad_h z\|_0$, is used here. In contrast, for $\vec{\alpha}_h\in \nabla\times CG_{k+1}(\Omega,\tau_h)$, we cannot establish a uniform inf-sup condition. As a simple example, take $\Omega=(0,1)^2$, with $\partial\Omega= \Gamma_1$, and consider $k=0$, so $\vec{\alpha}_{h}\in RT_1(\Omega,\tau_h)$. Consider a mesh such that one triangle has vertices $(0,0), \ (h,0)$, and $(0,h)$. Take $\vec{\alpha}_h$ to be nonzero only in this triangle, with value
\begin{equation}\label{inf0}
\vec{\alpha}_h=\nabla \times \left(1-\frac{x+y}{h}\right)=\frac{1}{h}\begin{bmatrix}
-1\\
1\\
\end{bmatrix}.
\end{equation} 
Clearly, $\nabla \cdot\vec{\alpha}_h=0$. Moreover, for any choice of $\vec{v}_h\in RT_{1}(\Omega,\tau_h)$ with $\vec{v}_h\cdot\vec{n}=0$ on $\partial\Omega$ and $u_h\in DG_{0}(\Omega,\tau_h)$, we have  $\int_{\Omega}\vec{\alpha}_h\cdot\vec{v}_h+\int_{\Omega}u\nabla\cdot\vec{\alpha}_h=0$, which results in a zero inf-sup constant.
Numerical experiments (not reported here) suggest that restricting the mesh so that no element has 3 vertices on the boundary yields an $\mathcal{O}(h)$ inf-sup constant. 

We next show that weakly implementing the essential boundary condition on $\vec{v}$ yields a discretization with $\mathcal{O}(1)$ continuity and coercivity constants and $\mathcal{O}(h)$ inf-sup constant, without any mesh restrictions but in a modified norm. This is slightly disadvantageous, because the error estimate loses some convergence due to both the sub-optimal inf-sup constant and the error analysis in the modified norm; however, we view the lack of restrictions on the mesh construction to be preferable to possible further results in the direction considered above.
To weakly impose the boundary condition, we will make use of a Nitsche-type penalty method. These approaches are based on adding three terms to the weak form, which are commonly denoted as the consistency, stability, and symmetry terms~\cite{Nitsche_1971,aziz}. 
%There is a natural tension between achieving symmetry in the resulting bilinear form and proving coercivity; we choose to make a non-symmetric modification to the bilinear form, for which coercivity is easily proven. 
Consider the case where $\partial\Omega = \Gamma_0 \cup \Gamma_1\cup \Gamma_3$, and modify the bilinear form ${a}\big((u,\vec{v}), (\phi, \vec{\psi})\big)$ from \eqref{28}, to be
\begin{alignat}{3}
\hat{a}\big((u_h,\vec{v}_h ), (\phi_h, \vec{\psi}_h)\big)+ \, &{b}\big((\phi_h, \vec{\psi}_h),\vec{\alpha}_h \big)&&={F}(\phi_h), \label{854}\\
&{b}\big((u_h,\vec{v}_h), \vec{\beta}_h\big)&&=0,\label{844}
\end{alignat}
$\forall (\phi_h,\vec{\psi}_h,\vec{\beta}_h)\in DG_{k}(\Omega,\tau_h)\times RT^{\Gamma_3}_{k+1}(\Omega,\tau_h)\times RT^{\Gamma_3}_{k+1}(\Omega,\tau_h) $,  where,
\begin{align}
\hat{a}\big((u_h,\vec{v}_h), (\phi_h, \vec{\psi}_h)\big) =& a\big((u_h,\vec{v}_h), (\phi_h, \vec{\psi}_h)\big)-\int_{\Gamma_1}\nabla\cdot\vec{v}_h\,\vec{\psi}_h\cdot\vec{n} \notag\\
&-\int_{\Gamma_1}\nabla\cdot\vec{\psi}_h\,\vec{v}_h\cdot\vec{n}+\frac{\lambda}{h}\int_{\Gamma_1}\vec{v}_h\cdot\vec{n}\,\vec{\psi}_h\cdot\vec{n} \label{eq:Nitsche_a_gamma1}.
\end{align}
Here, we impose the condition that $\vec{v}_h\cdot\vec{n}=0$ on $\Gamma_1$ directly by adding $\frac{\lambda}{h}\int_{\Gamma_1}\vec{v}_h\cdot\vec{n}\, \vec{\psi}_h\cdot\vec{n}$ to $a$ as defined in \eqref{28} for penalty parameter $\lambda>0$.
Consistent with the boundary condition, we add $-\int_{\Gamma_1}\nabla\cdot\vec{\psi}_h\,\vec{v}_h\cdot\vec{n}$ to the bilinear form.  For symmetry, we add the term $-\int_{\Gamma_1}\nabla\cdot\vec{v}_h\,\vec{\psi}_h\cdot\vec{n}$ to the bilinear form. With these Nitsche terms, we now prove well-posedness of the weak form, but using a modified norm for $\vec{v}_h$, defined in \eqref{eq:strengthened_norm}, which we recall here:
\begin{equation*}
\|\vec{v}_h\|^2_{\text{div}, \Gamma_1}=\|\vec{v}_h\|^2_{\text{div}}+h\|\nabla\cdot\vec{v}_h\|_{0,\Gamma_1}^2+\frac{1}{h}\|\vec{v}_h\cdot\vec{n}\|_{0,\Gamma_1}^2.
\end{equation*}

\begin{thm}\label{thm5}
	Let $\Omega \subset \mathbb{R}^d, \ d\in \{2,3\}$, with $\partial\Omega = \Gamma_0 \cup \Gamma_1\cup \Gamma_3$. Let $\{\tau_h\}$, $0<h\leq 1$, be a quasiuniform family of meshes of $\Omega$,
	and let $\lambda> 0$ be given.
	For sufficiently large $\lambda$, the weak form in \eqref{854}--\eqref{844}
	has a unique solution for $c_0\geq0, \ \ c_1>0$, and for $c_1=0$ if $\Gamma_0\cup \Gamma_1$ is nonempty.
\end{thm}
\begin{proof}
	As above, the existence and uniqueness of solutions follows from standard theory. We first show that the bilinear form $\hat{a}\big((u_h,\vec{v}_h), (\phi_h, \vec{\psi}_h)\big)$ is coercive for any pair $(u_h, \vec{v}_h)\in \eta_h$, where  
	\begin{align*}
	\eta_h=& \{(u_h,\vec{v}_h)\in DG_{k}(\Omega,\tau_h)\times RT^{\Gamma_3}_{k+1}(\Omega,\tau_h) | \\ & \phantom{xxxxxxxxx}  b\big( (u_h,\vec{v}_h),\vec{\alpha}_h\big)=0, \ \ \forall \vec{\alpha}_h\in RT^{\Gamma_3}_{k+1}(\Omega,\tau_h)\}.
	\end{align*}
	Since the strongly imposed boundary conditions for $\vec{v}$ and $\vec{\alpha}$ are identical, the kernel condition implies that $b\big((u_h,\vec{v}_h), \vec{v}_h\big) = 0$ for any pair $(u_h, \vec{v}_h)\in \eta_h$. Thus, any pair $(u_h, \vec{v}_h)\in \eta_h$ should satisfy $\|\vec{v}_h\|_0^2=-\int_{\Omega}u_h\nabla\cdot\vec{v}_h\leq \frac{1}{2}\left(\|u_h\|_0^2+\|\vec{v}_h\|_0^2\right)$. Employing the Cauchy-Schwarz inequality and letting $\gamma=\gamma_1(k)$ from Corollary~\ref{cor:inverse_trace_quasi}, we then have
	\begin{align}
    \label{eq:hata_coercivity}
	\hat{a}\big((u_h,\vec{v}_h), (u_h, \vec{v}_h)\big) =& c_0\|\vec{v}_h\|_0^2+ \|\nabla\cdot\vec{v}_h\|^2_0+c_1\|u_h\|^2_0-2\int_{\Gamma_1}\nabla\cdot\vec{v}_h\, \vec{v}_h\cdot\vec{n}+\frac{\lambda}{h}\|\vec{v}_h\cdot\vec{n}\|^2_{0,\Gamma_1} \nonumber \\
	\geq&\|\nabla\cdot\vec{v}_h\|_0^2+c_1\|u_h\|_0^2-\frac{h}{3\gamma}\|\nabla\cdot\vec{v}_h\|^2_{0,\Gamma_1}+\frac{\lambda-3\gamma}{h}\|\vec{v}\cdot\vec{n}\|^2_{0,\Gamma_1} \nonumber \\
	\geq&\frac{1}{3}\|\nabla\cdot\vec{v}_h\|_0^2+c_1\|u_h\|_0^2+\frac{h}{3\gamma}\|\nabla\cdot\vec{v}_h\|^2_{0,\Gamma_1}+\frac{\lambda-3\gamma}{h}\|\vec{v}\cdot\vec{n}\|^2_{0,\Gamma_1} \nonumber \\
	\geq & \frac{2\min\{1/3,c_1\}}{3}\|\vec{v}_h\|_0^2   +\frac{2\min\{1/3,c_1\}}{3}\big(\|\nabla\cdot\vec{v}_h\|^2_0+\|u\|^2_0\big) \nonumber \\
	&+\frac{h}{3\gamma}\|\nabla\cdot\vec{v}_h\|_{0,\Gamma_1}^2+\frac{\lambda-3\gamma}{h}\|\vec{v}_h\cdot\vec{n}\|^2_{0, \Gamma_1}.
	\end{align}
	Clearly, any choice of $\lambda>3\gamma$ completes the proof. %Taking, for example, $\lambda =\frac{2}{3}$ gives a coercivity constant of $\frac{2\min\{1,c_1\}}{3}$.
 When $\Gamma_0\cup\Gamma_1$ is nonempty and $c_0=0$, the coercivity proof is similar to the one in Theorem \ref{thm6}. 
	
	Continuity of $\hat{a}$, $b$, and ${F}$, can be established using Cauchy-Schwarz inequality. The resulting inequalities are that
	\begin{align*}
	\hat{a}\left(\left(u_h,\vec{v}_h\right),\left(\phi_h,\vec{\psi}_h\right)\right) &\leq  \|\hat{a}\|\|(u_h,\vec{v}_h)\|_{0,\text{div},\Gamma_1}\|(\phi_h,\vec{\psi}_h)\|_{0,\text{div},\Gamma_1},
	\end{align*}
	and
	\[
	b\left(\left(u_h,\vec{v}_h\right),\vec{\alpha}_h\right) \leq \|\left(u_h,\vec{v}_h\right)\|_{0,\text{div},\Gamma_1}\|\vec{\alpha}_h\|_{\text{div}},
	\]
	where $\|\hat{a}\|=3+c_0+c_1+\lambda, \ \text{and}\ \|\left(u_h,\vec{v}_h\right)\|_{0,\text{div},\Gamma_1}^2 = \|u_h\|_0^2 + \|\vec{v}_h\|_{\text{div},\Gamma_1}^2$. Thus, the continuity constant of the bilinear form $\hat{a}$ is $\mathcal{O}(1)$. 
        
	Finally, we consider the inf-sup condition, that $\exists \theta > 0$ such that
	\begin{align*}
	  I=&\sup_{\substack{(u_h,\vec{v}_h)\neq \in DG_{k}(\Omega,\tau_h)\times RT^{\Gamma_3}_{k+1}(\Omega,\tau_h)\\ (u_h,\vec{v}_h)\neq (0,\vec{0})}} \frac{\int_{\Omega}\vec{\alpha}_h\cdot\vec{v}_h+\int_{\Omega}u_h\nabla\cdot\vec{\alpha}_h}{ \sqrt{\|\vec{v}_h\|^2_{\text{div},\Gamma_1}+\|u_h\|^2_0}} \\
          &\geq \theta\|\vec{\alpha}_{h}\|_{\text{div}}, \ \forall \vec{\alpha}_h\in RT^{\Gamma_3}_{k+1}(\Omega,\tau_h).
	\end{align*}
	The choice $\vec{v}_h=\vec{\alpha}_h$ and $u_h=\nabla\cdot\vec{\alpha}_h$ and the inverse trace inequality from Corollary~\ref{cor:inverse_trace_quasi} implies that
	\begin{align*}
	I&\geq\frac{\|\vec{\alpha}_h\|_{\text{div}}^2}{\sqrt{\|\vec{\alpha}_h\|_{\text{div}}^2+h\|\nabla\cdot\vec{\alpha}_h\|_{0,\Gamma_1}^2+\frac{1}{h}\|\vec{\alpha}_h\cdot\vec{n}\|_{0,\Gamma_1}^2+\|\nabla\cdot\vec{\alpha}_{h}\|_0^2}}\\
	&\geq\frac{\|\vec{\alpha}_h\|_{\text{div}}^2}{\sqrt{\left(1+\frac{\gamma_1(k+1)}{h^2}\right)\|\vec{\alpha}_h\|_{0}^2+\left(2+\gamma_1(k)\right)\|\nabla\cdot\vec{\alpha}_h\|_{0}^2}} \\
	&\geq \frac{1}{\sqrt{ 1+\frac{\gamma_1(k+1)}{h^2} }}\|\vec{\alpha}_h\|_{\text{div}}= \frac{h}{\sqrt{ h^2+\gamma_1(k+1) }}\|\vec{\alpha}_h\|_{\text{div}}\geq \frac{h}{\sqrt{1+\gamma_1(k+1)}} \|\vec{\alpha}_h\|_{\text{div}}.
	\end{align*}
\end{proof}

While the above result establishes the existence and uniqueness of discrete solutions, we note that the inf-sup constant is $\mathcal{O}({h})$, owing to the contribution from the boundary terms in the formulation. \revise{We note that the following theorem relies on sufficient regularity of the continuum solution to establish the resulting error estimates; ensuring this regularity would require substantial assumptions on both the domain, $\Omega$, and the problem data, $f$ (and any non-homogeneous boundary data), that we do not elaborate here.}
%The following results apply standard theory (see, e.g., \cite{MR3097958}) to provide error estimates for the resulting discrete solutions.  
%For given functions $u$, $\vec{v}$, and $\vec{\alpha}$, we define the standard approximation properties of the discrete spaces as	\begin{eqnarray*}
%     E_{(u,\vec{v})}&=& \inf_{(\phi_h,\vec{\psi}_h)\in DG_k(\Omega,\tau_h)\times RT^{\Gamma_3}_{k+1}(\Omega,\tau_h)}\|(u,\vec{v})-(\phi_h,\vec{\psi}_h)\|_{0,\text{div},\Gamma_1},\\
%     E_{\vec{\alpha}}&=&\inf_{\vec{\beta}_h\in RT^{\Gamma_3}_{k+1}(\Omega,\tau_h)}\|\vec{\alpha}-\vec{\beta}_h\|_{{\text{div}}},
%	\end{eqnarray*}
%        where, as above, we use the product norm
%        $\|\left(u_h,\vec{v}_h\right)\|_{0,\text{div},\Gamma_1}^2 =
%        \|u_h\|_0^2 + \|\vec{v}_h\|_{\text{div},\Gamma_1}^2$.

\begin{thm}\label{cor1}
	Assume $\Omega\subset \mathbb{R}^d$, $d \in \{2,3\}$ is a bounded polygonal or polyhedral domain with Lipschitz boundary. Let the assumptions of Theorem \ref{thm5} be satisfied, and assume that $(u,\vec{v},\vec{\alpha})\in H^{k+5}(\Omega)\times [H^{k+4}(\Omega)]^d\times[H^{k+2}(\Omega)]^d$ is the solution of \eqref{454}--\eqref{444}.
	Let $(u_h,\vec{v}_h, \vec{\alpha}_h)$ be the unique solution of Problem  \eqref{854}-\eqref{844}.
	Then, there exists positive constants $M_1$ and $M_2$ such that
	\begin{align*}
		\|(u,\vec{v}) - (u_h,\vec{v}_h)\|_{0,\text{\normalfont div},\Gamma_1} & \leq M_1 h^{k}\bigg(|u|_{k+1}^2 + |\vec{v} |_{k+1}^2+|  \vec{v}|_{k+2}^2 + |\vec{\alpha} |_{k+1}^2+| \vec{\alpha}|_{k+2}^2\bigg)^{1/2}, \\
		\| \vec{\alpha} -\vec{\alpha}_h  \|_{\textrm{\normalfont div}} & \leq M_2 h^{k-1}\bigg(|u|_{k+1}^2 + |\vec{v} |_{k+1}^2+| \vec{v}|_{k+2}^2 + |\vec{\alpha} |_{k+1}^2+| \vec{\alpha} |_{k+2}^2\bigg)^{1/2}.
	\end{align*}
\end{thm}
\begin{proof}
%	Consider the constants
%	\begin{align*}
%	\|\hat{a}\| & = \max\{1,c_0,c_1,\lambda-3\gamma_1(k+1)\},\ \ \ \ 
%	\|b\|  = 1, \ \ \ \ \\
%	\zeta & = \frac{2\min\{1,c_1\}}{3},\ \ \ \
%	\gamma  = \frac{1}{\sqrt{C_\Omega(k+2)(k+d+1)}},
%	\end{align*}
%	which are the continuity, coercivity, and inf-sup constants from Theorem \ref{thm5}, rescaled so that the inequalities below explicitly show the dependence on $h$.
%	It follows from the coercivity condition that \cite[Inequality 2.11]{MR1115205},
%	\begin{align}
%	\|(u,\vec{v})-(u_h,\vec{v}_h)\|_{0,\text{\normalfont div},\Gamma_1}\leq&\left(1+\frac{\|\hat{a}\|}{h^{\frac{1}{2}}\zeta}\right) \inf_{(\phi_h,\vec{\psi}_h)\in \eta_h}\|(u,\vec{v})-(\phi_h,\vec{\psi}_h)\|_{0,\text{div},\Gamma_1}\notag\\&+\frac{\|b\|}{\zeta}\inf_{\vec{\beta}_h\in RT^{\Gamma3}_{k+1}(\Omega,\tau_{h})}\|\vec{\alpha}-\vec{\beta}_h\|_{\text{div}},\label{eq4.8}
%	\end{align}
  To prove the error estimates, we proceed in a similar way to \cite[Section II.2.1]{MR1115205}, with modifications to account for the use of Nitsche-type penalty methods to weakly impose the boundary conditions on $\vec{v}_h$.

  Coercivity of $\hat{a}$ over $\eta_h$, where $\eta_h$ is the set defined in the proof of Theorem \ref{thm5}, leads to the fact that for all $(\phi_h,\vec{\psi}_h)\in \eta_h$, we have
		\begin{align*}
	C\|(\phi_h,\vec{\psi}_h)-(u_h,\vec{v}_h)\|^2_{0,\text{div},\Gamma_1}\, & \leq \hat{a}\big((\phi_h-u_h,\vec{\psi}_h-\vec{v}_h), (\phi_h-u_h, \vec{\psi}_h-\vec{v}_h)\big)\\
	& = \hat{a}\big((\phi_h-u,\vec{\psi}_h-\vec{v}), (\phi_h-u_h, \vec{\psi}_h-\vec{v}_h)\big)\\
	& + \hat{a}\big((u-u_h,\vec{v}-\vec{v}_h), (\phi_h-u_h, \vec{\psi}_h-\vec{v}_h)\big).
	\end{align*}  
	        Here, $C$ is the coercivity constant, which is $\mathcal{O}(1)$. Note, first, that $\hat{a}\big((\phi_h-u,\vec{\psi}_h-\vec{v}), (\phi_h-u_h, \vec{\psi}_h-\vec{v}_h)\big)$	is continuous in its arguments, with $\mathcal{O}(1)$ continuity constant, so the first term is readily bounded.
                We next establish that we have enough regularity on the solution $(u, \vec{v}, \vec{\alpha})$ of the continuum problem to show that
\begin{equation}\label{eq:a_to_b}
\hat{a}\big((u-u_h,\vec{v}-\vec{v}_h), (\phi_h-u_h, \vec{\psi}_h-\vec{v}_h)\big)=b\left((\phi_h-u_h,\vec{\psi}_h-\vec{v}_h),\vec{\alpha}_h-\vec{\alpha}\right),
\end{equation}
where we note that the regularity is required to both integrate by parts and enforce relationships between the components of the continuum solution below.
To establish \eqref{eq:a_to_b}, note first (from the definition of $\hat{a}$) that
\begin{align*}
\hat{a}\big((u,\vec{v}), (\phi_h-u_h, \vec{\psi}_h-\vec{v}_h)\big)=&\int_{\Omega}\nabla\cdot\vec{v}\,\nabla\cdot(\vec{\psi}_h-\vec{v}_h)+c_0\int_{\Omega}\vec{v}\cdot(\vec{\psi}_h-\vec{v}_h)\\
&+c_1\int_{\Omega}u(\phi_h-u_h)-\int_{\Gamma_1}\nabla\cdot\vec{v}\, (\vec{\psi}_h-\vec{v}_h)\cdot\vec{n}\\
=&-\int_{\Omega}\left(\nabla\nabla\cdot\vec{v}-c_0\vec{v}\right)\cdot(\vec{\psi}_h-\vec{v}_h)+c_1\int_{\Omega}u(\phi_h-u_h)\\
=&-\int_{\Omega}\vec{\alpha}\cdot(\vec{\psi}_h-\vec{v}_h)+c_1\int_{\Omega}u(\phi_h-u_h),
\end{align*}
where we invoke Corollary~\ref{co} to ensure that $\vec{v} = \nabla u$, $\vec{\alpha} = \nabla\nabla\cdot\vec{v} -c_0\vec{v}$, and the boundary conditions on $\Gamma_0$ and $\Gamma_3$ ensure the other boundary integrals from integration by parts vanish.  Next, note that $\left(u,\vec{v}\right)$ satisfies~\eqref{454}; taking $\vec{\psi}=\vec{0}$ and $\phi=\phi_h-u_h\in L^2(\Omega)$ in~\eqref{454} gives
\begin{eqnarray*}
c_1\int_{\Omega}u(\phi_h-u_h)+\int_{\Omega}(\phi_h-u_h)\nabla\cdot\vec{\alpha}=F(\phi_h-u_h).
\end{eqnarray*}
Combining these, we have that
\begin{eqnarray*}
	\hat{a}\big((u,\vec{v}), (\phi_h-u_h, \vec{\psi}_h-\vec{v}_h)\big)&=&-\int_{\Omega}\vec{\alpha}\cdot(\vec{\psi}_h-\vec{v}_h)+F(\phi_h-u_h)-\int_{\Omega}(\phi_h-u_h)\, \nabla\cdot\vec{\alpha}\\
	&=&F(\phi_h-u_h)-b\left((\phi_h-u_h,\vec{\psi_h}-\vec{v}_h),\vec{\alpha}\right).
\end{eqnarray*}
On the other hand, we have, from \eqref{854},           $\hat{a}\big((u_h,\vec{v}_h), (\phi_h-u_h, \vec{\psi}_h-\vec{v}_h)\big)=F(\phi_h-u_h)-b\left((\phi_h-u_h,\vec{\psi_h}-\vec{v}_h),\vec{\alpha}_h\right)$.  Taking these together establishes~\eqref{eq:a_to_b}.

Now, for any $\vec{\beta}_h\in RT_{k+1}^{\Gamma_3}(\Omega,\tau_h)$,
\begin{align*}
  b\left((\phi_h-u_h,\vec{\psi}_h-\vec{v}_h),\vec{\alpha}-\vec{\alpha}_h\right) &= b\left((\phi_h-u_h,\vec{\psi}_h-\vec{v}_h),\vec{\alpha}-\vec{\beta}_h\right)\\
  &\leq \|(\phi_h-u_h,\vec{v}_h-\vec{\psi}_h)\|_{0,\text{div},\Gamma_1}\|\vec{\alpha}-\vec{\beta}_h\|_{\text{div}}
\end{align*}
since $(u_h-\phi_h,\vec{v}_h-\vec{\psi}_h) \in \eta_h$ and by the continuity of $b$. Thus,   
		\begin{align*}
	C\|(\phi_h,\vec{\psi}_h)-(u_h,\vec{v}_h)\|_{0,\text{div},\Gamma_1}\leq&\|\hat{a}\|\|(\phi_h,\vec{\psi}_h)-(u,\vec{v})\|_{0,\text{div},\Gamma_1}+\|\vec{\alpha}-\vec{\beta}_h\|_{\text{div}}.
	\end{align*}  
	  We choose $(\phi_h,\vec{\psi}_h)\in \eta_h$ to be the solution of the mixed Poisson problem posed on $DG_{k}(\Omega,\tau_h)\times RT^{\Gamma_3}_{k+1}(\Omega,\tau_h)$ with source term $-\Delta u$. Stability of this mixed formulation leads to the estimate
	\begin{align*}
		\inf_{(\phi_h,\vec{\psi}_h)\in\eta_h}\|(u,\vec{v})-(\phi_h,\vec{\psi_h})\|_{0,\text{div}}\, & \leq \inf_{(\phi_h,\vec{\psi_h})\in DG_{k}(\Omega,\tau_{h})\times RT^{\Gamma_3}_{k+1}}
		\|(u,\vec{v})-(\phi_h,\vec{\psi_h})\|_{0,\text{div}}\\
		& \leq \|(u,\vec{v})-(I^k_hu,\Pi_h^k\vec{v})\|_{0,\text{div}} \\
		& \leq r_1h^{k+1}\left(|u|^2_{k+1}+|\vec{v}|^2_{k+1}+|\vec{v}|^2_{k+2}\right)^{\frac{1}{2}}, 
		%&\leq& m_1h^{k+1}\bigg(|u|_{k+1}^2 + |\vec{v} |_{k+1}^2+|\nabla\cdot \vec{v}|_{k+1}^2\bigg)^{1/2},
	\end{align*}
        for constant $r_1$.
	Note, however, that our analysis is posed in the stronger norm, $\|(\cdot,\cdot)\|_{0,\text{div},\Gamma_1}$. To analyse the error in this norm, we have
	\begin{align}
	\inf_{(\phi_h,\vec{\psi}_h)\in\eta_h}\|(u,\vec{v})-(\phi_h,\vec{\psi_h})\|_{0,\text{div},\Gamma_1} \, & \leq \|(u,\vec{v})-(I^k_hu,\Pi_h^k\vec{v})\|_{0,\text{div}}+\sqrt{h}\|\nabla\cdot(\vec{v}-\Pi^k_h\vec{v})\|_{0,\Gamma_1}\notag\\
	&+\frac{1}{\sqrt{h}}\|\vec{v}-\Pi^k_h\vec{v}\|_{0,\Gamma_1},\notag \\
	& \revise{\leq r_2 \left(h^{k+1}+h^k\right)\bigg(|u|_{k+1}^2 + |\vec{v} |_{k+1}^2+|\vec{v}|_{k+2}^2\bigg)^{1/2}},\label{eq4.7}
	\end{align}
    where $r_2$ is a positive constant independent of $h$, and the terms above are bounded using Theorem~\ref{pro} and Lemmas~\ref{cor:inverse_trace_quasi_approx} and~\ref{Cor4}, resulting in degraded convergence rates.
    Finally, using the triangle inequality, \eqref{eq4.7}, and Theorem \ref{pro} leads to the $\mathcal{O}(h^k)$ estimate on $(u,\vec{v})$. To find the error estimate for $\vec{\alpha}$, we first use the triangle inequality, writing
	\begin{eqnarray}
	\|\vec{\alpha}-\vec{\alpha}_h\|_{\text{div}}\leq \|\vec{\alpha }-\vec{\beta}_h\|_{\text{div}}+\|\vec{\beta}_h-\vec{\alpha}_h\|_{\text{div}}.
	\end{eqnarray}
        Choosing $\vec{\beta}_h$ to be the interpolant of $\vec{\alpha}$ gives an
	error bound for the first term that is $\mathcal{O}(h^{k+1})$, as in Theorem~\ref{pro}. We then use the discrete inf-sup condition to bound the second term, writing $\gamma= \sqrt{h^2\left(1+\gamma_1(k+1)\right)+\gamma_1(k+1)}$,
	\begin{align*}
		\|\vec{\beta}_h-\vec{\alpha}_h\|_{\text{div}} \, & \leq \frac{\gamma}{h}\sup_{\substack{(\phi_h,\vec{\psi}_h)\in DG_{k}(\Omega,\tau_h)\times RT^{\Gamma_3}_{k+1}(\Omega,\tau_h)\\ (\phi_h,\vec{\psi}_h)\neq (0,\vec{0})}}\frac{b\left((\phi_h,\vec{\psi}_h),\vec{\beta}_h-\vec{\alpha}_h\right)}{\|(\phi_h,\vec{\psi}_h)\|_{0,\text{div},\Gamma_1}}\\
		& = \frac{\gamma}{h}\sup_{\substack{(\phi_h,\vec{\psi}_h)\in DG_{k}(\Omega,\tau_h)\times RT^{\Gamma_3}_{k+1}(\Omega,\tau_h)\\ (\phi_h,\vec{\psi}_h)\neq (0,\vec{0})}}\frac{b\left((\phi_h,\vec{\psi}_h),\vec{\alpha}-\vec{\alpha}_h\right)+b\left((\phi_h,\vec{\psi}_h),\vec{\beta}_h-\vec{\alpha}\right)}{\|(\phi_h,\vec{\psi}_h)\|_{0,\text{div},\Gamma_1}}\\
		& = \frac{\gamma}{h}\sup_{\substack{(\phi_h,\vec{\psi}_h)\in DG_{k}(\Omega,\tau_h)\times RT^{\Gamma_3}_{k+1}(\Omega,\tau_h)\\ (\phi_h,\vec{\psi}_h)\neq (0,\vec{0})}}\frac{\hat{a}\left((u-u_h,\vec{v}-\vec{v}_h),(\phi_h,\vec{\psi}_h)\right)+b\left((\phi_h,\vec{\psi}_h),\vec{\alpha}-\vec{\beta}_h\right)}{\|(\phi_h,\vec{\psi}_h)\|_{0,\text{div},\Gamma_1}}\\
		& \leq \frac{\gamma\|\hat{a}\|}{h}\|(u,\vec{v})-(u_h,\vec{v}_h)\|_{0,\text{div},\Gamma_1}+\frac{\gamma}{h}\|\vec{\alpha}-\vec{\beta}_h\|_{\text{div}},
	\end{align*}
        where we use the continuity of $\hat{a}$ and $b$ in the final inequality.
	The error estimate for $(u,\vec{v})$ above implies that the convergence rate of $\vec{\alpha}$ is $\mathcal{O}(h^{k-1})$.
\end{proof}

\section{Monolithic multigrid preconditioner}
\label{sec:multigrid}
We now consider the development of effective linear solvers for the resulting discretized systems.  We first consider the case where $\partial \Omega=\Gamma_0\cup\Gamma_2\cup\Gamma_3$ with constants $c_0, c_1 > 0$; however, the same arguments allow the case where $c_0=0$ if $\Gamma_2$ is empty. The discretizations above lead to block-structured linear systems that
can be written as
\begin{eqnarray}\label{system5.1}
\begin{bmatrix}
A_{11}&0& B_1^T  \\
0&A_{22}& B_2^T\\
B_1&B_2&0
\end{bmatrix}
\begin{bmatrix}
u\\\vec{v}
\\\vec{\alpha}
\end{bmatrix}=
\begin{bmatrix}
f_1\\f_2\\g
\end{bmatrix},
\end{eqnarray}
where $[u, \vec{v}, \vec{\alpha}]^T$ now refers to the vector of coefficients of the finite element basis functions.
For the weak form in \eqref{454}-\eqref{444}, $A_{11}$ is a mass matrix representing the discrete version of
the $L^2(\Omega)$ inner product on $DG_{k}(\Omega,\tau_h)$ weighted by $c_1$, $A_{22}$ is the
discrete version of the $H(\text{div}; \Omega)$ inner product on
$RT_{k+1}(\Omega,\tau_h)$ with weight $c_0$ on the $L^2(\Omega)$ term, $B_1$ is the weak gradient operator, and $B_2$ is
the $L^2(\Omega)$ inner product on $RT_{k+1}(\Omega,\tau_h)$. 

In order to efficiently solve such linear systems, we consider
preconditioned Krylov subspace methods. Two families of
preconditioners are popular for such block-structured problems.  Block factorization methods
\cite{MR2155549,MR4016136} approximate Gaussian elimination applied to the blocks of the discretization matrix, writing
\[
\mathcal{A}=\begin{bmatrix}
A& B^T\\ B&-C
\end{bmatrix}=\begin{bmatrix}
I& 0\\BA^{-1}&I 
\end{bmatrix}\begin{bmatrix}
A& 0\\0&S 
\end{bmatrix}\begin{bmatrix}
I& A^{-1}B^T \\0 &I 
\end{bmatrix},
\]
where $S=-(C+BA^{-1}B^T)$ is the Schur complement of $A$, assuming $A$ is invertible. Natural
block preconditioners are of block diagonal, $\textbf{P}_d$, and block
triangular, $\textbf{P}_t$ form, given as 
\[
\textbf{P}_d =\begin{bmatrix}
A&0\\ 0&\hat{S}
\end{bmatrix}, \ \ \ \ \textbf{P}_t =\begin{bmatrix}
A&B^T\\ 0&-\hat{S}
\end{bmatrix},
\]
where $\hat{S}$ is some approximation of $S$. The quality of these
preconditioners naturally depends on the approximation $\hat{S} \approx S$, and their
efficiency depends on the
availability of effective
fast solvers for the linear systems involving $A$ and
$\hat{S}$. Preliminary experiments in this direction revealed some
difficulties approximating the Schur complement in the presence of the Nitsche terms that would require
further investigation. Therefore, we focus on the development of efficient
monolithic multigrid preconditioners~\cite{MR1762024, MR3639325} in
this setting.

We use standard multigrid $V$-cycles with a direct solve on the
coarsest level (taken in the examples to be the mesh with $h=1/4$ for problems on unit-length
domains), and factor-2 coarsening between all grids. These cycles
are employed as preconditioners for FGMRES~\cite{MR1204241}.  We use standard
interpolation operators, partitioned based on the discretized
fields, of the form
\[
P=\begin{bmatrix}
I^k_h &&\\
&\Pi_h^{k+1}&\\
&&\Pi_h^{k+1}
\end{bmatrix},
\]
where the blocks $I^{k}_h$ and $\Pi^{k+1}_h $ are the natural
finite-element interpolation operators for the $DG_{k}(\Omega,\tau_h)$ and
$RT_{k+1}(\Omega,\tau_h)$ spaces, respectively. Coarse-grid operators are
formed by rediscretization, which is equivalent to a Galerkin
coarse-grid operator for constant $c_0, c_1 \in \mathbb{R}$.

The main challenge with monolithic multigrid methods is to develop an effective
relaxation method. In this work, as relaxation we make use of an additive overlapping Schwarz relaxation, which can be considered as a variant of the family of Vanka relaxation schemes originally proposed in
\cite{MR848451} to solve the saddle-point systems that arise from the
marker-and-cell (MAC) finite-difference discretization of the
Navier-Stokes equations.  Vanka relaxation methods encompass a variety
of overlapping multiplicative or additive Schwarz methods applied to
saddle-point problems, in which the subdomains are chosen so that the
corresponding subsystems are also saddle-point systems. 
Vanka-type relaxation has been used extensively for finite-element
discretizations, such as the discretizations arising from the Stokes
equations \cite{MR2840198}, magnetohydrodynamics \cite{adler2021}, and
liquid crystals \cite{MR3504546}.  Recently, a general-purpose
implementation of patch-based relaxation schemes, including Vanka relaxation, was provided in
\cite{Pcpatch}, which we employ via the finite-element discretization package Firedrake
\cite{rathgeber2017firedrake, kirby2018solver}.

Like other Schwarz methods, Vanka relaxation can be understood algebraically. Denoting
the set of all degrees of freedom in the problem by $Q$, we
partition $Q$ into $s$ overlapping subdomains or patches, $Q = \cup_{i=1}^s Q_i$,
and consider the stationary additive iteration with updates given by
\begin{equation*}
x\leftarrow x+ \sum_{i=1}^{s}R_i^T\mathcal{A}^{-1}_{ii}R_i(b-\mathcal{A}x), 
\end{equation*}    
where $\mathcal{A}x=b$ represents the linear system to be solved, 
$R_{i}$ is the injection operator from a global vector, $x$,
to a local vector, $x_{i}$, on $Q_i$ (with $R_{i}x=x_i$), and
$\mathcal{A}_{ii}=R_{i}\mathcal{A}R_{i}^T$ is the restriction of the
global system $\mathcal{A}$ to the degrees of freedom in $Q_i$.  While
inexact solution of the subdomain problems is relevant when the cardinality
of $Q_i$ is large, we consider small subdomain sizes, where direct
solution remains practical.  We construct the patches, $\{Q_i\}$,
topologically, as the so-called star patch around each vertex \cite{Pcpatch}
in the mesh, taking all degrees of freedom on vertex $i$, on edges adjacent to vertex
$i$, and on all cells adjacent to
vertex $i$ to form $Q_i$.  Figure \ref{fig:patches} shows the subdomain
construction around a typical vertex for the cases of discretization
using $DG_0$ and $RT_1$ elements (left) and using $DG_1$ and
$RT_2$ elements (right), noting that the $RT_k$ degrees of freedom
for both $\vec{v}$ and $\vec{\alpha}$ are included in the patch (and
are collocated on the mesh).  Rather than use the stationary iteration
given above, we use two steps of GMRES preconditioned by the Schwarz
method as the (pre- and post-) relaxation in the multigrid cycle on
each level.
\begin{figure}
	\begin{subfigure}{.5\textwidth}
		\centering
		\begin{tikzpicture}
		\filldraw [black] 
		(2/3,2/3) circle (2pt)
		(2/3,8/3) circle (2pt)
		(-2/3,4/3) circle (2pt)
		(4/3,4/3) circle (2pt)
		(-4/3,8/3) circle (2pt)
		(-2/3,10/3) circle (2pt)
		(-2/3,4/3) circle (2pt);
		\draw (0,0) -- (2,0) -- (0,2) -- (0,0);
		\draw (0,2) -- (2,2) -- (0,4) -- (0,2);
		\draw (2,0) -- (2,2);
		\draw (0,0) -- (0,2) -- (-2,2) -- (0,0);
		\draw (0,2) -- (0,4) -- (-2,4) -- (0,2);
		\draw (0,2) -- (0,4) -- (-2,4) -- (0,2);
		\draw (-2,2) -- (-2,4);
		\draw [->,line width=1.5] (1,1.7) -- (1,2.3);
		\draw [->,line width=1.5] (-1,1.7) -- (-1,2.3);
		\draw [->,line width=1.5] (-.3,1) -- (.3,1);
		\draw [->,line width=1.5] (-.3,3) -- (.3,3);
		\draw [->,line width=1.5] (.8,.8) -- (1.2,1.2);
		\draw [->,line width=1.5] (-1.2,2.8) -- (-.8,3.2);
		
		\end{tikzpicture}
		\label{fig:sfig1}
	\end{subfigure}%
	\begin{subfigure}{.5\textwidth}
		\centering
		
		\begin{tikzpicture}
		
		\filldraw [black] 
		(2/3,2/3) circle (2pt)
		(2/3,8/3) circle (2pt)
		(-2/3,4/3) circle (2pt)
		(4/3,4/3) circle (2pt)
		(-4/3,8/3) circle (2pt)
		(-2/3,10/3) circle (2pt)
		(-2/3,4/3) circle (2pt);
		\filldraw [black] 
		(.15,.15) circle (2pt)
		(.15,1.70) circle (2pt)
		(1.70,.15) circle (2pt)
		(.15,2.15) circle (2pt)
		(1.70,2.15) circle (2pt)
		(.15,3.70) circle (2pt)
		(1.85,.30) circle (2pt)
		(1.85,1.85) circle (2pt)
		(-.15,.30) circle (2pt)
		(-.15,1.85) circle (2pt)
		(-1.70,1.85) circle (2pt)
		(.3,1.85) circle (2pt)
		(-.3,2.15) circle (2pt)
		(-1.85,2.15) circle (2pt)
		(-1.85,3.70) circle (2pt)
		(-.15,2.30) circle (2pt)
		(-.15,3.85) circle (2pt)
		(-1.70,3.85) circle (2pt);
		(0.6,0.7) circle (4pt)
		(0.8,0.5) circle (4pt)
		(1.4,1.3) circle (4pt)
		(1.2,1.5) circle (4pt)
		5(.6,2.7) circle (4pt)
		(0.8,2.5) circle (4pt)
		(-0.6,1.3) circle (4pt)
		(-0.8,1.5) circle (4pt)
		(-1.4,2.7) circle (4pt)
		(-1.2,2.5) circle (4pt)
		(-0.6,3.3) circle (4pt)
		(-0.8,3.5) circle (4pt);
		\draw (.1,.1) -- (.1,.2) -- (.2,.2) -- (.2,.1);
		\draw (0,0) -- (2,0) -- (0,2) -- (0,0);
		\draw (0,2) -- (2,2) -- (0,4) -- (0,2);
		\draw (2,0) -- (2,2);
		\draw (0,0) -- (0,2) -- (-2,2) -- (0,0);
		\draw (0,2) -- (0,4) -- (-2,4) -- (0,2);
		\draw (0,2) -- (0,4) -- (-2,4) -- (0,2);
		\draw (-2,2) -- (-2,4);
		\draw [->,line width=1.5] (.9,1.7) -- (.9,2.3);
		\draw [->,line width=1.5] (1.1,1.7) -- (1.1,2.3);
		\draw [->,line width=1.5] (-.9,1.7) -- (-.9,2.3);
		\draw [->,line width=1.5] (-1.1,1.7) -- (-1.1,2.3);
		\draw [->,line width=1.5] (-.3,.9) -- (.3,.9);
		\draw [->,line width=1.5] (-.3,1.1) -- (.3,1.1);
		\draw [->,line width=1.5] (-.3,2.9) -- (.3,2.9);
		\draw [->,line width=1.5] (-.3,3.1) -- (.3,3.1);
		\draw [->,line width=1.5] (1.0,.6) -- (1.4,1.0);
		\draw [->,line width=1.5] (.8,.8) -- (1.2,1.2);
		
		\draw [->,line width=1.5] (-1.0,2.6) -- (-.6,3.0);
		\draw [->,line width=1.5] (-1.2,2.8) -- (-.8,3.2);
		\draw [fill, draw=black] (.45,.45) rectangle (.6 ,.6);
		\draw [fill, draw=black] (.6,.6) rectangle (.75,.75);
		\draw [fill, draw=black] (1.55,1.55) rectangle (1.40,1.40);
		\draw [fill, draw=black] (1.4,1.4) rectangle (1.25,1.25);
		\draw [fill, draw=black] (-.45,1.55) rectangle (-.6,1.40);
		\draw [fill, draw=black] (-.6,1.40) rectangle (-.75,1.25);
		\draw [fill, draw=black] (.45,2.45) rectangle (.6,2.6);
		\draw [fill, draw=black] (.6,2.6) rectangle (.75,2.75);
		\draw [fill, draw=black] (-.45,3.55) rectangle (-.6,3.40);
		\draw [fill, draw=black] (-.6,3.40) rectangle (-.75,3.25);
		\draw [fill, draw=black] (-1.55,2.45) rectangle (-1.40,2.60);
		\draw [fill, draw=black] (-1.40,2.60) rectangle (-1.25,2.75);
		\end{tikzpicture}

		\label{fig:sfig2}
	\end{subfigure}
	\caption{Star patches for $DG_{0}-RT_{1}$(left) and $DG_{1}-RT_{2}$(right) discretizations. Filled discs denote $DG$ degrees of freedom, while arrows and filled squares denote edge and interior $RT$ degrees of freedom, respectively.}
	\label{fig:patches}
\end{figure}
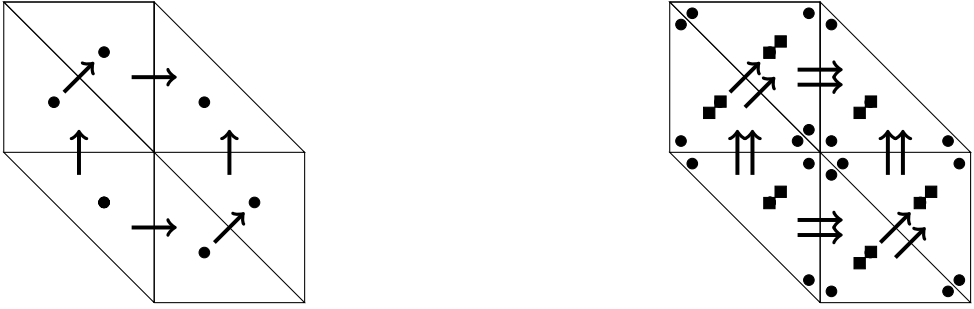 

For the case where $\partial\Omega=\Gamma_0\cup\Gamma_1$ with $c_0=c_1=0$, a modification of the above solver framework is needed. Note that, in this case, while the linear system is well-posed by Theorem \ref{thm5}, the modified weak form in \eqref{854}-\eqref{844} has the same structure as \eqref{system5.1} except that $A_{11}$ becomes the zero matrix and Nitsche boundary terms appear in $A_{22}$. 
The approach above performs poorly in this case, as might be expected, particularly with the zero block for $A_{11}$.  To overcome this, we adopt the idea of preconditioning the resulting discretization matrix using an auxiliary operator that corresponds to the discretization of another PDE, related to the inner products in which the PDE is analyzed \cite{mardal2011,kirby2010}.  Here, since the biharmonic operator is equivalent to the norm in Lemma \ref{lem:hilbert}, we add a scaling factor times the $L^2(\Omega)$ inner product in the (1,1) block of the auxiliary operator. Preliminary experiments (reported below in Table \ref{bih}) indicate that choosing the scaling factor to be $\mathcal{O}\left(h^{-1}\right)$ gives better performance than $\mathcal{O}(1)$ values.

\section{Numerical experiments}
\label{sec:numerical}
In this section, we present numerical experiments to measure
finite-element convergence rates and demonstrate the performance of
the proposed monolithic multigrid method, stopping when either the residual norm or its relative reduction is less than $10^{-8}$.  These numerical results
were calculated using the finite-element discretization package
Firedrake \cite{rathgeber2017firedrake}, which offers close
integration with PETSc for the linear solvers \cite{balay2018petsc,
	kirby2018solver}.  The relaxation scheme is implemented
using the PCPATCH framework \cite{Pcpatch}.  All numerical
experiments were run on a workstation with dual 8-core Intel Xeon 1.7
GHz CPUs and 384 GB of RAM.  For reproducibility, the codes used to
generate the numerical results, as well as the major Firedrake
components needed, have been archived on Zenodo \cite{zenodo/Firedrake-20221011.0}.
To measure solution quality, we make use of the method of manufactured
solutions, prescribing forcing terms and boundary data to exactly
match those of a known solution, $u_{ex}$.  Taking uniform meshes, as
described below, with representative mesh size $h$, we define $u_h$ to
be the finite-element solution on the mesh, and define the
approximation error $e_h = u_{ex} - u_h$.  With this, we can define
the relative error in the $L^2(\Omega)$ norm on mesh $h$ as $R_e(h)=\frac{\|e_h\|_0}{\|u_{ex}\|_0}$. As needed, we extend these definitions to other quantities, such as
the $L^2(\Omega)$ error in $\vec{v}$ and $\vec{\alpha}$, the error
in $\vec{v}$ and $\vec{\alpha}$ in the $H(\text{div})$ (semi-)norm, and the
error in any boundary terms included in the norms used above. 

\subsection{2D Experiments}
In $2D$, we consider experiments on the unit square using uniform
``right'' triangular meshes (Figure \ref{erd}, left) and on the
L-shaped domain with vertices $(0,0),\ (0,1),\ (\frac{1}{2},
1),\ (\frac{1}{2},\ \frac{1}{2}),\ (1,\frac{1}{2}),\ \text{and}\ (1,0)$
using uniform ``crossed'' triangular meshes (Figure \ref{erd},
right). The discretizations
proposed here have larger numbers of degrees of freedom and nonzeros
in their matrices than are typically encountered with Lagrange
elements for second-order problems. We therefore record the matrix
dimensions, $N$, and number of nonzeros, nnz, in Table \ref{tabl} for
the discretizations on the unit square, for several levels of
refinement, $h$, and orders of the discretization, $k$.
In all figures, we use blue, red, and green lines to present results for $k=0,1,2$, respectively. 
\begin{figure}
	\centering
	\includegraphics[width=0.43\linewidth]{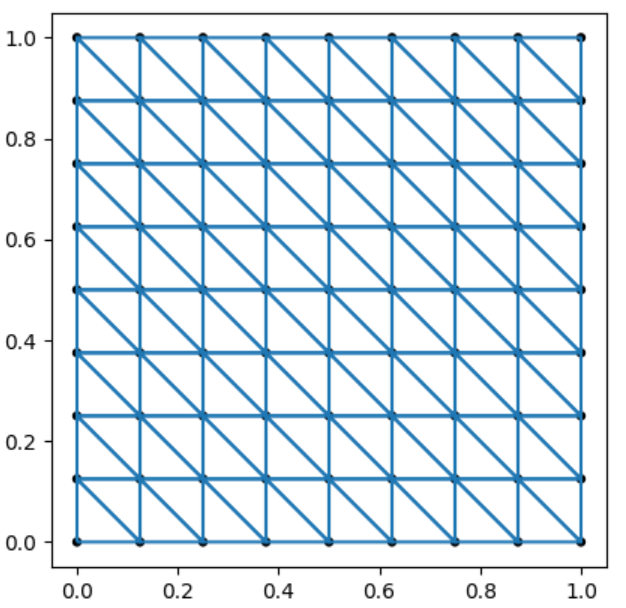}
	\includegraphics[width=0.43\linewidth]{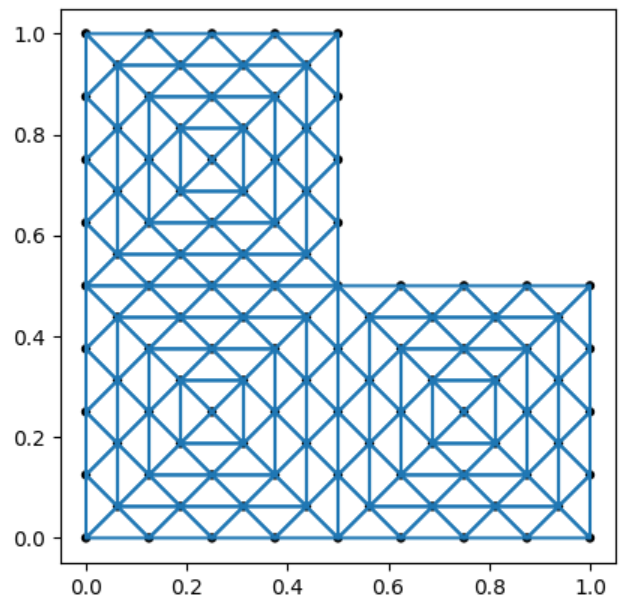}
	\caption{Left: unit square domain with uniform right triangular
		mesh ($h=\frac{1}{8}$). Right: L-shaped domain with uniform
		crossed triangular mesh $(h=\frac{1}{8})$.} \label{erd}
\end{figure}
\begin{table}
	\centering
	\caption{Dimension, $N$, and the number of nonzeros, nnz, in the system
		matrix for $u\in DG_k(\Omega,\tau_h), \ \vec{v}\in
		RT_{k+1}(\Omega,\tau_h),\ \vec{\alpha}\in RT_{k+1}(\Omega,\tau_h)$ on uniform meshes of
		the unit square domain in 2D.}
	\begin{tabular}{c |c c| c  c|c  c}  
		\toprule
		\multicolumn{1}{c|} {} &\multicolumn{2}{c|} {$k=0$}& \multicolumn{2}{c|}{ $k=1$}& \multicolumn{2}{c}{$k=2$}\\ \midrule
		$1/h$&$N$&nnz&$N$&nnz& $N$&nnz  \\
		$2^6$&33,024  &352,768 & 107,008&2,762,752 & 221,952&10,179,072 \\
		$2^7$&131,584   &1,410,048 &427,008 & 11,046,912& 886,272&40,707,072 \\
		$2^8$&525,312 &5,638,144 &1,705,984& 44,179,456 &3,542,016 &162,809,856 \\
		$2^9$&2,099,200 &22,548,480 &6,819,840 &176,701,440 &14,161,920  &651,202,560\\ \bottomrule
	\end{tabular}\label{tabl}
\end{table}

\revisenew{
We present two-dimensional numerical experiments in three parts.  First, in Subsection~\ref{smooth}, we investigate the convergence of our finite-element methods for sufficiently smooth exact solutions (as covered by our analysis), including the classical biharmonic problem with the clamped boundary conditions (requiring the use of the Nitsche penalty method).  Then, in Subsection~\ref{nonsmooth}, we consider the case of exact solutions, $u_{ex}\notin H^4(\Omega)$, where convergence rates require further analysis. We demonstrate that, in certain cases, our method does not converge to the exact solution of the biharmonic equation if that solution does not lie in $H^4(\Omega)$ (i.e.~in some cases it is not merely a matter of degraded convergence rates). Finally, in Subsection~\ref{Solver}, we provide a comparison between direct solvers and the multigrid-preconditioned FGMRES algorithm for $H^2$ elliptic problems with $c_0 \ge 0$ and $c_1 > 0$, as well as multigrid-preconditioned FGMRES iteration counts for the biharmonic problem with clamped boundary conditions, which is challenging due to the inclusion of the Nitsche boundary terms.} 

\subsubsection{\revisenew{Smooth solutions}}  \label{smooth}
\revisenew{Here, we consider the unit square domain with two exact solutions, one that is smooth (in $C^{\infty}(\Omega)$), $u_{1ex}=\sin(2\pi x)\cos(3\pi y)$, and one that is in $H^4(\Omega)$, but not $H^5(\Omega)$, $u_{2ex}=\left(\sin(2\pi x)+x^{\frac{9}{2}}\right)\left(\cos(3\pi y)+y^{\frac{17}{4}}\right)$.
In the results that follow, we plot}
$\log(R_e(\cdot,h))$ against $\log_2(1/h)$, so that the slopes of the lines represent the experimentally measured convergence rates for $(u,\vec{v},\vec{\alpha})
\in DG_k (\Omega,\tau_h)\times RT_{k+1}(\Omega,\tau_h) \times RT_{k+1}(\Omega,\tau_h)$, for $k\in\{0,1,2\}$.
\revisenew{In the examples in this subsection, filled discs denote the measured $(u,\vec{v})$ error in the $L^2\times H(\text{div})$ norm, and squares denote the error in $\vec{\alpha}$ measured in the $H(\text{div})$ norm.}

Let $\partial \Omega=\Gamma_N\cup \Gamma_S\cup\Gamma_E\cup\Gamma_W$, meaning the North, South, East, and West faces of the square.  Figure \ref{tab:ex1_uv} presents results for the problem with $c_0 = 0$ and $c_1 = 1$ with boundary
$\partial\Omega = \Gamma_0\cup \Gamma_3$, where $\Gamma_0=\Gamma_E\cup\Gamma_W$, and $\Gamma_3=\Gamma_N\cup\Gamma_S$. 
Figure \ref{tab:ex2_uv} presents results when $c_0 = 2$, $c_1 = 4$, and $\partial\Omega=\Gamma_0\cup\Gamma_2\cup\Gamma_3$ with $\Gamma_0=\Gamma_E\cup\Gamma_W$, $\Gamma_2=\Gamma_S$, and $\Gamma_3=\Gamma_N$.  Since we omit $\Gamma_1$ from these examples, there is no need to use the Nitsche boundary terms considered in that case.   We note that we see optimal convergence for all $k$ with
$(u, \vec{v})$ in the $L^2\times H(\text{div})$ norm and
$\vec{\alpha}$ in the $H(\text{div})$ norm for the smooth exact solution, $u_{1ex}$, on the left of these figures. Considering the $H^4(\Omega)$ solution, $u_{2ex}$, on the right, we see optimal convergence for small $k$, but degraded performance for $k=2$, where the lack of smoothness in $\alpha$ is reflected in the numerical results. These results are  consistent with the analysis in Corollaries \ref{cor4} and \ref{thm4}, although we note that the $H^4(\Omega)$ case outperforms the expected convergence from the analysis. 
\begin{figure}
	\begin{subfigure}{.5\textwidth}
		\centering
		
		\begin{tikzpicture}[scale=.75]
		\begin{semilogyaxis}[
		xlabel={$\log_2(1/h)$},
				xtick={6,7,8,9},
		ylabel={$\log\left(R_e(\cdot,h)\right)$},
		]
		\addplot[
		mark=*,mark options={scale=1,solid},color=seabornblue
		]
		coordinates {
			(6, 0.04168660182971993)
			(7, 0.02084600736406262)
			(8, 0.010423342251395887)
			(9, 0.005211713455081824)
		};
		\addplot[
		mark=square,mark options={scale=2,solid},color=seabornblue
		]
		coordinates {
			(6, 0.041674122665203685)
			(7, 0.0208444467964996)
			(8, 0.010423146921620175)
			(9, 0.005211688911884107)
		};
		\addplot[dash pattern=on 4pt off 6pt on 8pt off 8pt, domain=6:9] {8/(2^x)};
		\legend
		{,,$m=-1$,,,$m=-2$,,,$m=-3$}
		\addplot[
		mark=*,mark options={scale=1,solid},color=seabornred
		]
		coordinates {			
			(6, 0.0009934042545686696)
			(7, 0.0002484407758985374)
			(8, 6.211581428655616e-05)
			(9, 1.5529306149895575e-05)
		};
		\addplot[
		mark=square,mark options={scale=2,solid},color=seabornred
		]
		coordinates {			
			(6, 0.000993361204551166)
			(7,0.0002484380691961439)
			(8,  6.211563951158168e-05)
			(9, 1.552929378853e-05)
		};
		\addplot[dashed, domain=6:9] {10/(2^x)^2};
		\addplot[
		mark=*,mark options={scale=1,solid},color=seaborngreen
		]
		coordinates {			
			(6, 1.7204514705148833e-05)
			(7, 2.151353615878747e-06)
			(8, 2.6894389343578745e-07)
			(9, 3.3618760241465344e-08)
		};
		\addplot[
		mark=square,mark options={scale=2,solid},color=seaborngreen
		]
		coordinates {			
			(6, 1.7204324347483587e-05)
			(7, 2.1513474008521868e-06)
			(8, 2.6894389343578745e-07)
			(9, 3.3618760241465344e-08)
		};		
		
		\addplot[dotted, domain=6:9]{12/(2^x)^3};
		
		\end{semilogyaxis}
		\end{tikzpicture}
		
	\end{subfigure}%
	\begin{subfigure}{.5\textwidth}
		\centering
		\begin{tikzpicture}[scale=.75]
		\begin{semilogyaxis}[
		xlabel={$\log_2(1/h)$},
				xtick={6,7,8,9},
		]
		
		\addplot[
		mark=*,mark options={scale=1,solid},color=seabornblue
		]
		coordinates {

			(6, 0.04418416336933884)
			(7, 0.022095624009133265)
			(8, 0.011048254039679838)
			(9, 0.005524182258325259)
		};
		
		\addplot[ mark=square,mark options={scale=2,solid},color=seabornblue
		]
		coordinates {		
			(6, 0.04304926707886965)
			(7, 0.021532783713325822)
			(8, 0.010767421194979759)
			(9, 0.005383842799496321)	
			
		};
		\addplot[dash pattern=on 4pt off 6pt on 8pt off 8pt, domain=6:9] {8/(2^x)};
		\addplot[
		mark=*,mark options={scale=1,solid},color=seabornred
		]
		coordinates {		
			(6, 0.001101087733505842)
			(7, 0.00027537810213724815)
			(8, 6.885126400090261e-05)
			(9, 1.7213251115044502e-05)	
			
		};
		
		\addplot[
		mark=square,mark options={scale=2,solid},color=seabornred
		]
		coordinates {		
			(6, 0.001052032676333308)
			(7, 0.000263247412433017)
			(8, 6.589646536220787e-05)
			(9, 1.653802701194845e-05)	
			
		};
		\addplot[dashed, domain=6:9] {12/((2^x)^2)};
		\addplot[
		mark=*,mark options={scale=1,solid},color=seaborngreen
		]
		coordinates {		
			(6, 1.9524209570861605e-05)
			(7, 2.441625577514291e-06)
			(8, 3.0547135111364527e-07)
			(9, 3.8520883573613846e-08)	
			
		};

		\addplot[
		mark=square,mark options={scale=2,solid},color=seaborngreen
		]
		coordinates {		
			(6, 1.875559585901144e-05)
			(7, 3.0365457940047264e-06)
			(8, 1.1921349012360247e-06)
			(9, 1.1921349012360247e-06)	
			
		};
		\addplot[dotted, domain=6:9] {12/((2^x)^3)};
		\legend
		{,,$m=-1$,,,$m=-2$,,,$m=-3$}
		\end{semilogyaxis}
		\end{tikzpicture}
	\end{subfigure}
	\caption{Relative approximation errors and rate of convergence
		for the unit square domain with $c_0=0, \ c_1=1, \ \partial\Omega = \Gamma_0\cup \Gamma_3$ and $(u,\vec{v},\vec{\alpha}) \in DG_k(\Omega,\tau_h)\times RT_{k+1}(\Omega,\tau_h)\times RT_{k+1}(\Omega,\tau_h)$, $k\in{0,1,2}$. Blue, red, and green lines present results for $k=0,1,2$, respectively. \revise{Filled discs denote the measured $(u,\vec{v})$ error in the $L^2\times H(\text{div})$ norm, and squares denote the error in $\vec{\alpha}$ measured in the $H(\text{div})$ norm.  Reference lines of slopes -1, -2, and -3 approximate convergence rates.} Left: smooth solution $u_{ex} = u_{1ex}$. Right: rough solution $u_{ex} = u_{2ex}$.}\label{tab:ex1_uv}
\end{figure}
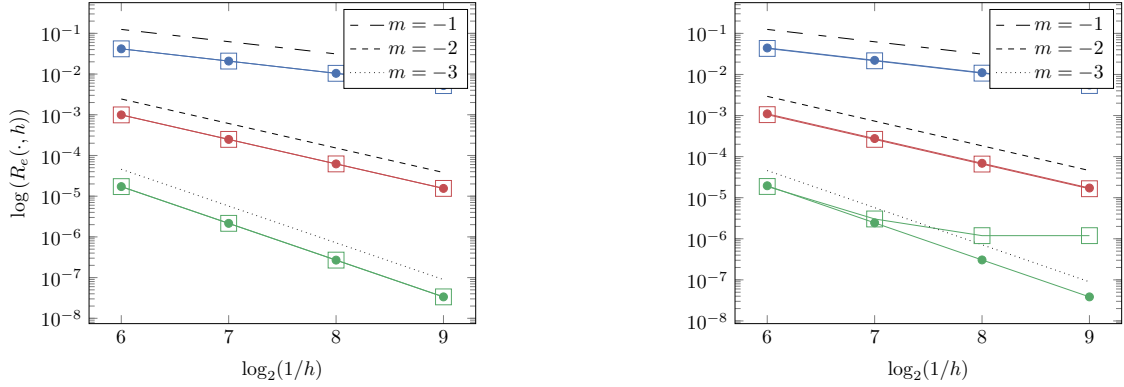

\begin{figure}
	\begin{subfigure}{.5\textwidth}
		\centering
		\begin{tikzpicture}[scale=.75]
		\begin{semilogyaxis}[
		xlabel={$\log_2(1/h)$},
				xtick={6,7,8,9},
		ylabel={$\log\left(R_e(\cdot,h)\right)$},
		]
		\legend
		{,,$m=-1$,,,$m=-2$,,,$m=-3$}
		\addplot[
		mark=*,mark options={scale=1,solid},color=seabornblue
		]
		coordinates {	
			(6, 0.041748737706979774)
			(7, 0.02085378979439833)
			(8, 0.010424315480717431)
			(9, 0.0052118351202492695)
		};
		
		\addplot[
		mark=square,mark options={scale=2,solid},color=seabornblue
		]
		coordinates {
			
			(6, 0.04167412298752674)
			(7, 0.020844446836977962)
			(8, 0.010423146926685854)
			(9, 0.0052116889125174955)
		};
		\addplot[dash pattern=on 4pt off 6pt on 8pt off 8pt, domain=6:9] {8/(2^x)};
		\addplot[
		mark=*,mark options={scale=1,solid},color=seabornred
		]
		coordinates {
			
			(6, 0.001010682242966361)
			(7, 0.0002506103218890227)
			(8, 6.238757052638014e-05)
			(9, 1.5563309339865328e-05)
		};
		\addplot[
		mark=square,mark options={scale=2,solid},color=seabornred
		]
		coordinates {
			
			(6, 0.0009933612044264427)
			(7, 0.00024843806918863683)
			(8, 6.211563951112455e-05)
			(9, 1.5529293788503894e-05)
		};
		\addplot[dashed, domain=6:9] {10/(2^x)^2};
		\addplot[
		mark=*,mark options={scale=1,solid},color=seaborngreen
		]
		coordinates {
			
			(6, 1.7339235424124668e-05)
			(7, 2.1597897723449984e-06)
			(8, 2.694716582590668e-07)
			(9, 3.365175472512711e-08)
		};
		\addplot[
		mark=square,mark options={scale=2,solid},color=seaborngreen
		]
		coordinates {
			
			(6, 1.720432434699534e-05)
			(7, 2.1513474016989567e-06)
			(8, 2.6894372702654923e-07)
			(9, 3.3626073053322385e-08)
		};
		\addplot[dotted, domain=6:9] {14/(2^x)^3};
		\end{semilogyaxis}
		\end{tikzpicture}
	\end{subfigure}
	\begin{subfigure}{.5\textwidth}
		\centering
		\begin{tikzpicture}[scale=.75]
		\begin{semilogyaxis}[
		xlabel={$\log_2(1/h)$},
				xtick={6,7,8,9},
		]
		
		\addplot[mark=*,mark options={scale=1,solid},color=seabornblue
		]
		coordinates {	
			
			(6, 0.04438968851900649)
			(7, 0.02212628739684407)
		    (8, 0.011052700024501929)
		    (9, 0.005524814220965205)
		};
		
		\addplot[
		mark=square,mark options={scale=2,solid},color=seabornblue
		]
		coordinates {
			
			(6, 0.04307088537116145)
			(7, 0.021543604738027314)
			(8, 0.01077283304389388)
			(9, 0.005386548835723917)
		};
		\addplot[dash pattern=on 4pt off 6pt on 8pt off 8pt, domain=6:9] {8/(2^x)};
		\addplot[
		mark=*,mark options={scale=1,solid},color=seabornred
		]
		coordinates {
			
			(6, 0.0011260822960482968)
			(7, 0.00027851901240584097)
			(8, 6.924469883804892e-05)
			(9, 1.726247039295428e-05)
		};
		
		\addplot[
		mark=square,mark options={scale=2,solid},color=seabornred
		]
		coordinates {	
			(6, 0.0010529529097721678)
			(7, 0.00026347540821209845)
			(8, 6.595187992460276e-05)
			(9, 1.6550313711828696e-05)
		};
		\addplot[dashed, domain=6:9] {12/(2^x)^2};
		\addplot[
		mark=*,mark options={scale=1,solid},color=seaborngreen
		]
		coordinates {	
			(6, 1.9693967872993604e-05)
			(7, 2.4559358821431164e-06)
			(8, 3.103736745825803e-07)
			(9, 4.3841250245177606e-08)
		};
		
		\addplot[
		mark=square,mark options={scale=2,solid},color=seaborngreen
		]
		coordinates {	
			(6, 1.8766419159722312e-05)
			(7, 3.020057918535396e-06)
			(8, 1.1761154937221932e-06)
			(9, 6.707475007673204e-07)
		};
		\addplot[dotted, domain=6:9] {14/(2^x)^3};
		\legend
		{,,$m=-1$,,,$m=-2$,,,$m=-3$}
		\end{semilogyaxis}
		\end{tikzpicture}	
	\end{subfigure}%
	\caption{Relative approximation errors and rates of convergence
		for the unit square domain with $c_0=2, \ c_1=4, \ \partial\Omega = \Gamma_0\cup\Gamma_2\cup \Gamma_3$ and $(u,\vec{v},\vec{\alpha}) \in DG_k(\Omega,\tau_h)\times RT_{k+1}(\Omega,\tau_h)\times RT_{k+1}(\Omega,\tau_h)$, $k\in{0,1,2}$. Blue, red, and green lines present results for $k=0,1,2$, respectively. \revise{Filled discs denote the measured $(u,\vec{v})$ error in the $L^2\times H(\text{div})$ norm, and squares denote the error in $\vec{\alpha}$ measured in the $H(\text{div})$ norm.  Reference lines of slopes -1, -2, and -3 approximate convergence rates.} Left: smooth solution $u_{ex} = u_{1ex}$. Right: rough solution $u_{ex} = u_{2ex}$.}\label{tab:ex2_uv}
\end{figure}
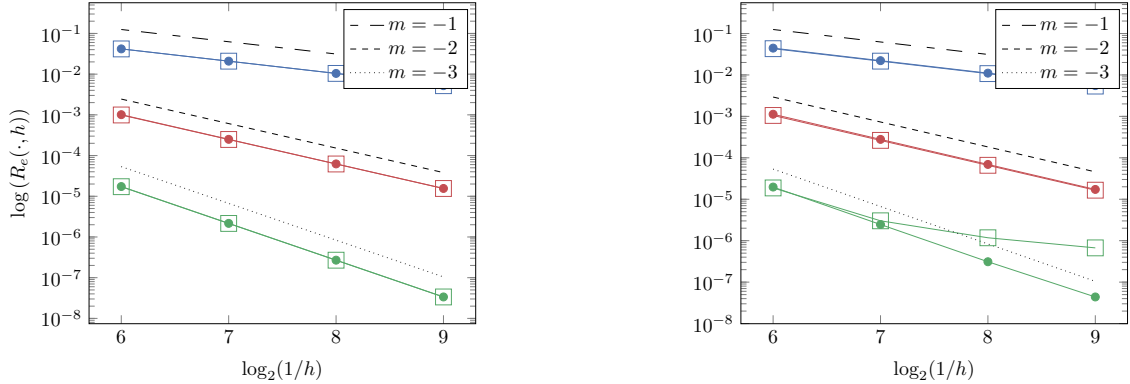

We next consider the classical biharmonic problem.  Figure \ref{tab:ex3_uv*} shows
results for the case where the exact solution is $u_{1ex}$ and $\partial\Omega = \Gamma_0\cup\Gamma_1\cup\Gamma_2$, and we take $(u,\vec{v},\vec{\alpha}) \in
DG_2(\Omega,\tau_h)\times RT_3(\Omega,\tau_h)\times RT_3(\Omega,\tau_h)$.  For this choice ($k=2$), Corollary \ref{cor1} gives an
expected convergence rate of $\mathcal{O}(h^2)$ for $(u,\vec{v})$ and of $\mathcal{O}(h)$ for
$\vec{\alpha}$.  In contrast with that result, we observe almost $\mathcal{O}(h^{5/2})$ for $(u,\vec{v})$ in the modified $L^2\times H(\text{div})$ norm and
$\mathcal{O}(h^{3/2})$ convergence for $\vec{\alpha}$. 
\begin{rmk}
	{For the right-triangular meshes given  in Figure~\ref{erd} (Left), $\gamma_1(2)\approx 41$ and $\gamma_1(3)\approx 62$. Therefore, we choose $\lambda=125, \ \text{and} \ 210$ respectively for the numerical results given in Figure~\ref{tab:ex3_uv*} and Table~\ref{bih}. Note that, for both cases, $\lambda> 3\gamma_1$, which is necessary to satisfy the coercivity condition \eqref{eq:hata_coercivity} for the bilinear form $\hat{a}$.}
\end{rmk}
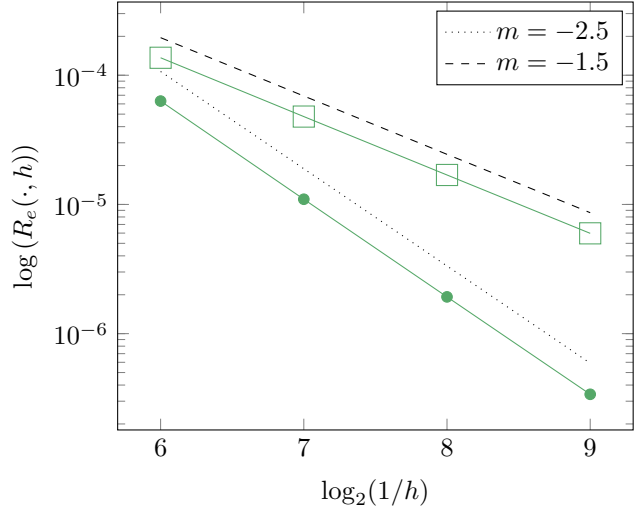
\begin{figure}
		\centering
		\begin{tikzpicture}
			\begin{semilogyaxis}[
				ylabel={$\log\left(R_e(\cdot,h)\right)$},
				xlabel={$\log_2(1/h)$},
				xtick={6,7,8,9},
				]
				\legend
				{,,$m=-2.5$,$m=-1.5$}
				\addplot[
				mark=*,mark options={scale=1,solid},color=seaborngreen
				]
				coordinates {

					(6, 6.3092502443371e-05)
					(7, 1.097198040062105e-05)
					(8, 1.9257655609904473e-06)
					(9, 3.393170335258995e-07)
				};
				
				\addplot[
				mark=square,mark options={scale=2,solid},color=seaborngreen
				]
				coordinates {

					(6, 0.0001365333378814537)
					(7, 4.787028335706131e-05)
					(8, 1.690466015705936e-05)
					(9, 5.975721616429765e-06)
					
				};
				\addplot[dotted, domain=6:9] {3.5/(2^x)^(5/2)};
				\addplot[dashed, domain=6:9] {.1/(2^x)^(3/2)};

			\end{semilogyaxis}
		\end{tikzpicture}
	\caption{ Relative approximation errors and rate of convergence
		for the biharmonic problem in the unit square domain and $ \partial\Omega=\Gamma_0\cup\Gamma_1\cup\Gamma_3$. Filled discs denote the measured $(u,v)$ error in the $L^2\times H(\text{div})$ norm, and squares denote the error in $\alpha$ measured in the $H(\text{div})$ norm.  \revisenew{Reference lines of slopes $-3/2$ and $-5/2$ approximate convergence rates.} The exact solution is $C^\infty$, $u_{ex}= u_{1ex}$. }\label{tab:ex3_uv*}
\end{figure}

%\begin{table}
%	\centering

%	\caption{Relative approximation errors and rate of convergence
%		for the unit square domain with $c_0=0, \ c_1=1, \ \partial\Omega = \Gamma_0$ and $(u,\vec{v},\vec{\alpha}) \in DG_0(\Omega,\tau_h)\times RT_1(\Omega,\tau_h)\times RT_1(\Omega,\tau_h)$.}\label{tab:ex1_uv}

%	\begin{tabular}{|| c|| c| c c| c c  ||} 	
%		\hline
%		$\frac{1}{h}$  & $R_e(u)$&$R_e(\vec{v})$ &$R_e(\nabla\cdot\vec{v})$&$R_e(\vec{\alpha})$&$R_e(\nabla\cdot\vec{\alpha})$\\ [0.5ex] 
%		\hline
%		$2^6$ & 4.17  $\times 10^{-2}$ & 3.87 $\times 10^{-2}$  & 4.17 $\times 10^{-2}$  & 3.87  $\times 10^{-2}$  & 4.17$\times 10^{-2}$   \\ 
%		\hline
%		$2^7$& 2.09   $\times 10^{-2}$ & 1.94 $\times 10^{-2}$  &2.09 $  \times 10^{-2}$ &1.94 $\times 10^{-2}$  &2.09 $\times 10^{-2}$ \\
%		\hline
%		$2^8$ & 1.04 $\times 10^{-2}$&9.68 $\times 10^{-3}$ & 1.04$\times 10^{-2}$ &9.68 $\times 10^{-3}$  & 1.04 $\times 10^{-2}$ \\
%		\hline
%		$2^9$ & 5.21$\times 10^{-3}$&  4.84$\times 10^{-3}$  & 5.21 $\times 10^{-3}$& 4.84  $\times 10^{-3}$ & 5.21 $\times 10^{-3}$ \\
%		\hline

%		$R_c$ &1.00 &1.00 &1.00 &1.00 &1.00 \\  
%		\hline 
%	\end{tabular}

%\end{table} 
\subsubsection{\revisenew{Nonsmooth solutions}}\label{nonsmooth}
\revisenew{In this section, we consider examples that do not conform to the theory presented above, with a variety of exact solutions that are less regular than $u\in H^4(\Omega)$.  To gain insight into the structure of the errors in these cases, we present errors separately for each of $u$ (measured in the $L^2$ norm), $\vec{v}$, and $\vec{\alpha}$ (both measured in the $H(\text{div})$ norm), with triangles denoting measured errors in $u$, asterisks denoting measured errors in $\vec{v}$, and squares denoting the error in $\vec{\alpha}$.}

\revisenew{In Figure \ref{tab:ex_uvalpha}, we consider the unit square domain with nonsmooth solution, $u_{3ex}=\left(\sin(2\pi x)+x^{\frac{7}{2}}\right)\left(\cos(3\pi y)+y^{\frac{13}{4}}\right)\in H^3(\Omega)$, noting that $u_{3ex}\notin H^4(\Omega)$.  Figure~\ref{tab:ex_uvalpha} presents convergence results for
$(u,\vec{v},\vec{\alpha})
  \in DG_k (\Omega,\tau_h)\times RT_{k+1}(\Omega,\tau_h) \times RT_{k+1}(\Omega,\tau_h)$, for $k\in\{0,1,2\}$, with values of $c_0$ and $c_1$ and boundary conditions matching those considered in Figure~\ref{tab:ex1_uv} (at left) and Figure~\ref{tab:ex2_uv} (at right).  Table~\ref{Aziz1} summarizes the measured convergence rates from this data, using the points at $\log_2(1/h)= 8 \text{ and } 9$. While our analysis does not cover this case, we observe that our mixed formulation converges suboptimally to the correct solution.  For the lowest-order discretization, we see convergence that is relatively good for all three components of the solution (but still somewhat suboptimal compared to the $\mathcal{O}(h)$ solution error we might hope to see), while we see diverging approximations of $\alpha$ for $k=1$ and $2$, with convergence rates of $\mathcal{O}(h^{1/4})$ for $u$ and $\vec{v}$.  It is natural to speculate that this occurs since the solution is in $H^{13/4-\epsilon}(\Omega)$ for any $\epsilon>0$, but we leave analysis of this case for future work.}
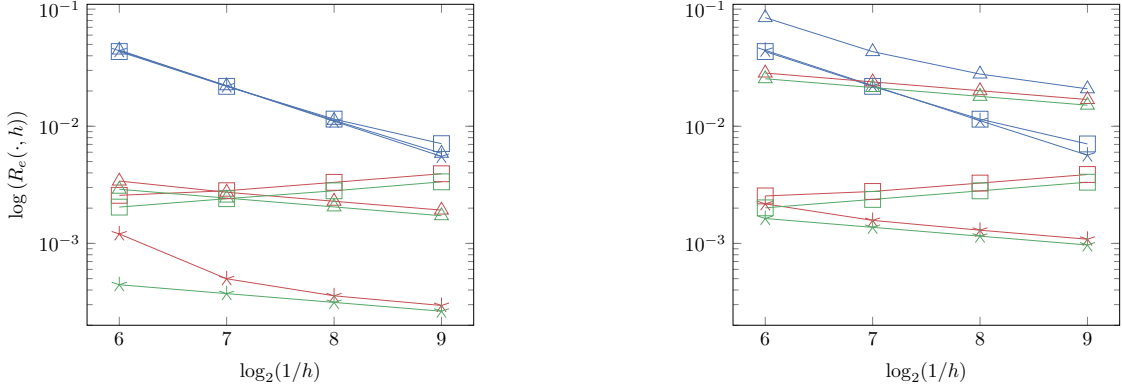
\begin{figure}
	\begin{subfigure}{.5\textwidth}
		\centering		
		\begin{tikzpicture}[scale=.75]
			\begin{semilogyaxis}[
				xlabel={$\log_2(1/h)$},
				xtick={6,7,8,9},
				ylabel={$\log\left(R_e(\cdot,h)\right)$},
				ymin=2e-4, ymax =1.1e-1,
				legend pos=outer north east
				]
				\addplot[
				mark=triangle,mark options={scale=2,solid},color=seabornblue
				]
				coordinates {
					(6, 0.04471871831053931)
					(7, 0.022167298611009276)
					(8, 0.011229237650515356)
					(9, 0.0058788404261674535)
				};
				\addplot[
				mark=star,mark options={scale=2,solid},color=seabornblue
				]
				coordinates {
					(6, 0.04415620630786615)
					(7, 0.022084606598908353)
					(8, 0.011047334176785403)
					(9, 0.005530385118526125)
				};
				\addplot[
				mark=square,mark options={scale=2,solid},color=seabornblue
				]
				coordinates {
					(6, 0.04343012122910707)
					(7, 0.021919037040669477)
					(8, 0.011509186160586536)
					(9, 0.007104643445756075)
				};
				\addplot[
				mark=triangle,mark options={scale=2,solid},color=seabornred
				]
				coordinates {			
					(6, 0.0034052488347851954)
					(7, 0.0027341177596181734)
					(8, 0.00228914780567003)
					(9,0.001924264738504105)
				};
				\addplot[
				mark=star,mark options={scale=2,solid},color=seabornred
				]
				coordinates {			
					(6, 0.0012044948977306344)
					(7, 0.0004988115122908802)
					(8, 0.00035686151561466564)
					(9, 0.00029499116381895144)
				};
				\addplot[
				mark=square,mark options={scale=2,solid},color=seabornred
				]
				coordinates {			
					(6, 0.0025729757040022297)
					(7,  0.0028163038713433716)
					(8, 0.0033212303323637755)
					(9, 0.0039311901068505815)
				};

				\addplot[
				mark=triangle,mark options={scale=2,solid},color=seaborngreen
				]
				coordinates {			
					(6, 0.0028992413199129164)
					(7, 0.0024359303672613825)
					(8, 0.002048086128112288)
					(9, 0.0017222697634819168)
				};
				\addplot[
				mark=star,mark options={scale=2,solid},color=seaborngreen
				]
				coordinates {			
					(6, 0.00044368990898592314)
					(7, 0.0003727485671804347)
					(8, 0.00031345080508541186)
					(9, 0.00026358938414988616)
				};	
				\addplot[
				mark=square,mark options={scale=2,solid},color=seaborngreen
				]
				coordinates {			
					(6, 0.0020415262123906)
					(7, 0.002408792712010754)
					(8, 0.0028124761360657754)
					(9, 0.0033539054942268907)
				};	
				
			\end{semilogyaxis}
		\end{tikzpicture}
		
	\end{subfigure}%
	\begin{subfigure}{.5\textwidth}
		\centering
		\begin{tikzpicture}[scale=.75]
			\begin{semilogyaxis}[
				xlabel={$\log_2(1/h)$},
				xtick={6,7,8,9},
				ymin=2e-4, ymax =1.1e-1,
				legend pos=outer north east
				]
				\addplot[
				mark=triangle,mark options={scale=2,solid},color=seabornblue
				]
				coordinates {
					(6, 0.0846095643843088)
					(7, 0.04336801751720005)
					(8, 0.02791910231198683)
					(9, 0.0208739760097631)
				};
				\addplot[
				mark=star,mark options={scale=2,solid},color=seabornblue
				]
				coordinates {
					(6, 0.0445562528827176)
					(7, 0.022239751492658093)
					(8, 0.011168369484044422)
					(9, 0.005670412360249303)
				};
				\addplot[
				mark=square,mark options={scale=2,solid},color=seabornblue
				]
				coordinates {
					(6, 0.04344518709104595)
					(7, 0.021920758343949995)
					(8, 0.011494522098725285)
					(9, 0.007063475024300224)
				};
				\addplot[
				mark=triangle,mark options={scale=2,solid},color=seabornred
				]
				coordinates {			
					(6, 0.02847942484485712)
					(7, 0.023920800430080994)
					(8, 0.020113018034149478)
					(9, 0.01691306636429596)
				};
				\addplot[
				mark=star,mark options={scale=2,solid},color=seabornred
				]
				coordinates {			
					(6, 0.002174699603682373)
					(7, 0.0015721295726885526)
					(8, 0.0012960400009328218)
					(9, 0.0010856113279670621)
				};
				\addplot[
				mark=square,mark options={scale=2,solid},color=seabornred
				]
				coordinates {			
					(6, 0.002541705863636511)
					(7, 0.002775781270190391)
					(8, 0.0032731912183909806)
					(9, 0.003874361372064378)
				};

				\addplot[
				mark=triangle,mark options={scale=2,solid},color=seaborngreen
				]
				coordinates {			
					(6, 0.025456029737045745)
					(7, 0.021406887135325313)
					(8, 0.018001811069538104)
					(9, 0.015138131226683272)
				};
				\addplot[
				mark=star,mark options={scale=2,solid},color=seaborngreen
				]
				coordinates {			
					(6, 0.001631899091579343)
					(7, 0.001371305584669851)
					(8, 0.0011527477460225313)
					(9, 0.0009692175096843715)
				};	
				\addplot[
				mark=square,mark options={scale=2,solid},color=seaborngreen
				]
				coordinates {			
					(6, 0.0020119423692414092)
					(7, 0.002373937918447346)
					(8, 0.0028099991203916603)
					(9, 0.0033299283283319413)
				};	
				
			\end{semilogyaxis}
		\end{tikzpicture}
	\end{subfigure}
	\caption{Relative approximation errors for the unit square domain with nonsmooth solution, $u_{3ex}\in H^3(\Omega)$.  Triangles denote the measured error in $u$ in the $L^2$ norm, stars denote the measured error in $\vec{v}$ in the $H(\text{div})$ norm, and squares denote the measured error in $\vec{\alpha}$ in the $H(\text{div})$ norm. Blue, red, and green lines present results for $k=0,1,2$, respectively. Left: $c_0=0$, $c_1 =1$, and $\partial\Omega = \Gamma_0\cup\Gamma_3$. Right: $c_0= 2$, $c_1 = 4$, and $\partial\Omega=\Gamma_0\cup\Gamma_2\cup\Gamma_3$.}\label{tab:ex_uvalpha}
\end{figure}
\begin{table}
	\centering
	\caption{Experimental convergence rates for the unit square domain with exact solution $u_{ex}= \left(\sin(2\pi x)+ x^{7/2}\right)\left(\cos(3\pi y)+y^{13/4}\right)$, and $(u,\vec{v},\vec{\alpha})\in DG_{k}\times RT_{k+1}\times RT_{k+1}$, $k\in\{0,1,2\}$. Errors in $u$ are measured in the $L^2$ norm and errors in $\vec{v}$ and $\vec{\alpha}$ are measured in the $H(\text{div})$ norm. Dashes mean that we do not observe convergence.} \label{Aziz1}
	\begin{tabular}{l|*{2}{c}}\toprule
		\backslashbox{$k$}{}
		&\makebox{Figure \ref{tab:ex_uvalpha} (left)}&\makebox{Figure \ref{tab:ex_uvalpha} (right). }
		\\\midrule
		$k=0$   &(0.93, 1.00, 0.70)&(0.42, 0.98, 0.70) \\
		$k=1$	&(0.25, 0.27, -)&(0.26, 0.24, -) \\
		$k=2$	&(0.25, 0.25, -)&(0.25, 0.25, -) \\\bottomrule
	\end{tabular}
\end{table}

\revise{In the next experiments, we consider the L-shaped domain with $\partial\Omega=\Gamma_0$ and discretizations with $k\in \{0,1,2\}$.  Here, we consider harmonic functions as exact solutions, given in polar coordinates as $u_{p} = r^{2p/3} \sin(2p\theta/3)$, for $p\in \{1,2,4\}$, and impose nonhomogeneous boundary conditions to match the solution along the boundary \revisenew{(enforced weakly, by including the boundary integrals generated when integrating by parts on Equations \eqref{eq:graddiv_v} and \eqref{eq:grad_u} to get Equations \eqref{eq:weak_psi} and \eqref{3.16}, respectively, with known values for $\nabla\cdot \vec{v}_p = \Delta u_p$ and $u_p$)}.  We place the origin for the $(r,\theta)$ coordinates at the re-entrant corner at $(1/2,1/2)$, and measure the angle counterclockwise from the positive $y$-direction from this origin (see Figure~\ref{erd} (right)).  This results in homogeneous Dirichlet BCs for $u$ along the two edges incident to the re-entrant corner, but non-homogeneous values on the other four edges of the domain. Note that $(u_p,\vec{v}_p)$ are in $H^{1+2p/3-\epsilon}(\Omega)\times [H^{2p/3-\epsilon}(\Omega)]^2$~\cite[Theorem 1.2.18]{grisvard1992singularities}, for any $\epsilon>0$, \revisenew{but not in $H^{1+2p/3}(\Omega)\times [H^{2p/3}(\Omega)]^2$}, with $\vec{\alpha}_p =\vec{0}$. Thus, we might expect to correctly compute $\vec{\alpha}_p$, since it is trivially represented in the $RT_{k+1}$ space for all $k\geq 0$, but expect, at best, to obtain $\mathcal{O}(h^{\min\{1+2p/3, k+1\}})$ and $\mathcal{O}(h^{\min\{2p/3, k+1\}})$ convergence rates for $u_p$ in the $L^2$ norm and $\vec{v}_p$ in the $H(\text{div})$ norm, respectively.  Considering that the error analysis above is done in the $L^2\times H(\text{div})$ product norm, it may even be reasonable to expect no better than $\mathcal{O}(h^{\min\{2p/3, k+1\}})$ convergence rates for both $u_p$ and $\vec{v}_p$.}

\revise{From most to least smooth, we show results for this problem for $p=4$ in Figure \ref{tab:ex3_uv} and for $p=2$ (at right) and $p=1$ (at left) in Figure \ref{tab:ex4_uv}.  In all cases, we omit reporting results for $\vec{\alpha}_p=\vec{0}$, as the absolute error is always at the level of machine precision.  Table~\ref{Aziz} summarizes convergence rates, measuring the slopes of this data using the points at $\log_2(1/h)= 8 \text{ and } 9$. For $p=2$ and $p=4$, we observe best possible convergence rates for both $u$ and $\vec{v}$, while we observe the best possible convergence rate for $\vec{v}$, but slightly less than this for $u$, when $p=1$.  
	(Preliminary numerical experiments, not reported here, showed that multiplying a harmonic function by a smooth function, to yield nonzero $\vec{\alpha}$ resulted in similarly reduced convergence orders for $\vec{\alpha}$, including a lack of convergence when taking $k=0$).  \revisenew{Clearly for $p=1$, $u_p\notin H^2(\Omega)$, so our proposed method is converging to a function that cannot be a true solution of the biharmonic problem, but is admitted by the weaker mixed formulation that requires only $u\in L^2(\Omega)$ and $\vec{v} = \nabla u \in H(\text{div};\Omega)$.  This is confirmed by computing the smallest eigenvalue of this discretization of the biharmonic, which yields a value of approximately 1,486 on a mesh with $h=1/128$ discretized with $u \in DG_1(\Omega,\tau_h)$ and $\vec{v},\vec{\alpha}\in RT_2(\Omega,\tau_h)$, close to the square of the smallest Laplacian eigenvalue and not the true smallest biharmonic eigenvalue of approximately 2,641 (cf., \cite{MR3630795}). This is reminiscent of the Sapondjan paradox, where solving the biharmonic problem via a mixed formulation involving two Poisson problems gives the wrong solution on nonconvex domains~\cite[pg.~268]{MR1283387}.} Whether anything can be proven about the convergence to solutions with degraded regularity \revisenew{(meaning $u\in H^2(\Omega)$, but not $H^4(\Omega)$)} remains an open question for future research \revisenew{that is clearly complicated by the weaker spaces used in the mixed formulation}.}
\begin{figure}
		\centering
		\begin{tikzpicture}
			\begin{semilogyaxis}[
				xlabel={$\log_2(1/h)$},
				xtick={6,7,8,9},
				ylabel={$\log\left(R_e(\cdot,h)\right)$},
				]
				\addplot[
				mark=triangle,mark options={scale=2,solid},color=seabornblue
				]
				coordinates {
					(6, 0.019151705397277643)
					(7, 0.009575973147087736)
					(8, 0.004788001625837589)
					(9, 0.002394002694347563)
				};
				\addplot[
				mark=star,mark options={scale=2,solid},color=seabornblue
				]
				coordinates {
					(6, 0.019844983981633642) 
					(7, 0.009923399929995629)
					(8, 0.004961825684210949)
					(9, 0.002480930085384756)
				};

				\addplot[
				mark=triangle,mark options={scale=2,solid},color=seabornred
				]
				coordinates {			
					(6, 0.00012607151332429757)
					(7, 3.15178646458411e-05)
					(8, 7.879465322509811e-06)
					(9, 1.969866278639525e-06)
				};
				\addplot[
				mark=star,mark options={scale=2,solid},color=seabornred
				]
				coordinates {			
					(6, 6.723479343133009e-05)
					(7, 1.6811826546100174e-05)
					(8, 4.203277988365137e-06)
					(9, 1.0508526433350594e-06)
				};

				\addplot[
				mark=triangle,mark options={scale=2,solid},color=seaborngreen
				]
				coordinates {			
					(6, 2.84997375964379e-07)
					(7, 3.563026405833877e-08)
					(8, 4.454060714216331e-09)
					(9, 5.56771375863909e-10)
				};
				\addplot[
				mark=star,mark options={scale=2,solid},color=seaborngreen
				]
				coordinates {			
					(6, 3.679826571795428e-07)
					(7, 5.8296768590712765e-08)
					(8, 9.215042026712448e-09)
					(9, 1.4546342260683967e-09)
				};	
				
			\end{semilogyaxis}
		\end{tikzpicture}
	\caption{Relative approximation errors for the biharmonic problem on the L-shaped domain with $\partial\Omega = \Gamma_0$ with triangles denoting the measured $u$ error in the $L^2$ norm, and stars denoting the error in $\vec{v}$ measured in the $H(\text{div})$ norm with $u_{ex} = u_{4}$. Blue, red, and green lines present results for $k=0,1,2$, respectively.}\label{tab:ex3_uv}
\end{figure}
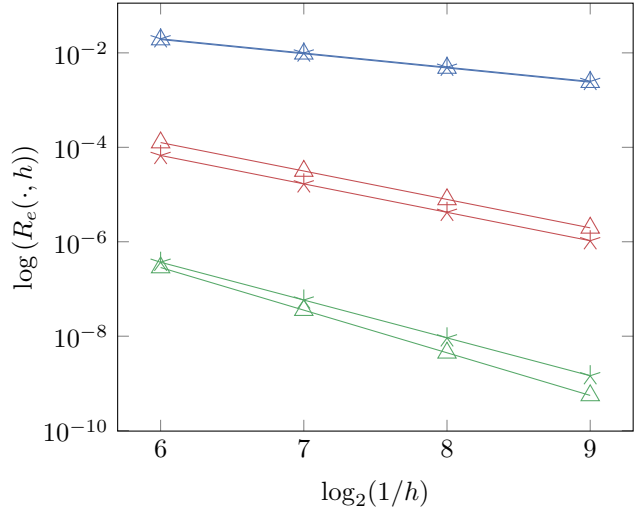

		\begin{figure}
			\begin{subfigure}{.5\textwidth}
				\centering		
				\begin{tikzpicture}[scale=.75]
					\begin{semilogyaxis}[
						xlabel={$\log_2(1/h)$},
				xtick={6,7,8,9},
						ylabel={$\log\left(R_e(\cdot,h)\right)$},
						%				ymin=1e-14, ymax =3,
						legend pos=outer north east
						]
						\addplot[
						mark=triangle,mark options={scale=2,solid},color=seabornblue
						]
						coordinates {
							(6, 0.006811184242226288)
							(7, 0.003399666597531204)
							(8, 0.0016978623591794493)
							(9, 0.0008482897178196133)
						};
						\addplot[
						mark=star,mark options={scale=2,solid},color=seabornblue
						]
						coordinates {
							(6, 0.02949680869734019)
							(7, 0.01866792969233651)
							(8, 0.011794407490004594)
							(9, 0.0074435141397885445)
						};
						\addplot[
						mark=triangle,mark options={scale=2,solid},color=seabornred
						]
						coordinates {			
							(6, 0.0001374662188108578)
							(7, 5.2441180533936956e-05)
							(8, 2.025584110914861e-05)
							(9, 7.894975525145211e-06)
						};
						\addplot[
						mark=star,mark options={scale=2,solid},color=seabornred
						]
						coordinates {			
							(6, 0.012366685919134366)
							(7, 0.0077907158393263325)
							(8, 0.004907891064923884)
							(9, 0.0030917895035194065)
						};

						\addplot[
						mark=triangle,mark options={scale=2,solid},color=seaborngreen
						]
						coordinates {			
							(6, 5.1276726943612194e-05)
							(7, 1.9989945556031786e-05)
							(8, 7.841838363230986e-06)
							(9, 3.0890181148715895e-06)
						};
						\addplot[
						mark=star,mark options={scale=2,solid},color=seaborngreen
						]
						coordinates {			
							(6, 0.007756488624046035)
							(7, 0.004886328240298008)
							(8, 0.003078205552284382)
							(9, 0.0019391508978769257)
						};	
						
					\end{semilogyaxis}
				\end{tikzpicture}
				
			\end{subfigure}%
			\begin{subfigure}{.5\textwidth}
				\centering
				\begin{tikzpicture}[scale=.75]
					\begin{semilogyaxis}[
						xlabel={$\log_2(1/h)$},
				xtick={6,7,8,9},
						ylabel={$\log\left(R_e(\cdot,h)\right)$},
						%		ymin=1e-14, ymax =3,
						legend pos=outer north east
						]
						\addplot[
						mark=triangle,mark options={scale=2,solid},color=seabornblue
						]
						coordinates {
							(6, 0.011160984504385418)
							(7, 0.005580501623862287)
							(8, 0.0027902520865161966)
							(9, 0.0013951262108123713)
						};
						\addplot[
						mark=star,mark options={scale=2,solid},color=seabornblue
						]
						coordinates {
							(6, 0.006400709012936982) 
							(7, 0.0032149398702156556)
							(8, 0.0016120452328319671)
							(9, 0.0008074600821214635)
						};

						\addplot[
						mark=triangle,mark options={scale=2,solid},color=seabornred
						]
						coordinates {			
							(6, 2.369741022885573e-05)
							(7, 5.951004238206241e-06)
							(8, 1.491927420177557e-06)
							(9, 3.7363752742424024e-07)
						};
						\addplot[
						mark=star,mark options={scale=2,solid},color=seabornred
						]
						coordinates {			
							(6, 0.00032960153032767983)
							(7, 0.00013083956442770146)
							(8, 5.1929558536295235e-05)
							(9, 2.060917906640518e-05)
						};
						
						\addplot[
						mark=triangle,mark options={scale=2,solid},color=seaborngreen
						]
						coordinates {			
							(6, 8.057432305487513e-07)
							(7, 1.5994872659275176e-07)
							(8, 3.1744330899205454e-08)
							(9, 6.299514774064254e-09)
						};
						\addplot[
						mark=star,mark options={scale=2,solid},color=seaborngreen
						]
						coordinates {			
							(6, 0.0001175965427661822)
							(7, 4.666821986410038e-05)
							(8, 1.8520295367036994e-05)
							(9, 7.3497840889249145e-06)
						};

					\end{semilogyaxis}
				\end{tikzpicture}
			\end{subfigure}
			\caption{Relative approximation errors for the biharominc problem on the L-shaped domain with $\partial\Omega = \Gamma_0$ with triangles denoting the measured $u$ error in the $L^2$ norm, and stars denoting the error in $\vec{v}$ measured in the $H(\text{div})$ norm. Blue, red, and green lines present results for $k=0,1,2$, respectively. Left: $u_{ex}= u_1$. Right: $u_{ex}= u_2$.}\label{tab:ex4_uv}
		\end{figure}
		
		\begin{table}
			\centering
			\caption{Experimental convergence rates for the L-shaped domain with $\partial\Omega=\Gamma_0$ and $(u,\vec{v},\vec{\alpha})\in DG_{k}\times RT_{k+1}\times RT_{k+1}$, $k\in\{0,1,2\}$. The exact solutions are $u_p = r^{2p/3}\sin(2p\theta/3)$, $p=1,2,4$. Errors in $u$ are measured in the $L^2$ norm and errors in $\vec{v}$ are measured in the $H(\text{div})$ norm. Results for $\vec{\alpha}_p=\vec{0}$ are omitted as it is exactly approximated by this discretization. } \label{Aziz}
			\begin{tabular}{l|*{3}{c}}\toprule
				\backslashbox{$k$}{p}
				&\makebox{Figure \ref{tab:ex4_uv} (left, $p=1$)}&\makebox{Figure \ref{tab:ex4_uv} (right, $p=2$)}&\makebox{Figure \ref{tab:ex3_uv} ($p=4$)}
				\\\midrule
				$k=0$   &(1.00, 0.66)&(1.00, 1.00) &(1.00, 1.00) \\
				$k=1$	&(1.36, 0.67)&(2.00, 1.33) &(2.00, 2.00) \\
				$k=2$	&(1.34, \revisenew{0.67})&(2.33, 1.33)&(3.00, 2.66) \\\bottomrule
			\end{tabular}
		\end{table}
		
\subsubsection{\revisenew{Monolithic multigrid solvers}}\label{Solver}
\revisenew{Finally, we report on the effectiveness of the monolithic multigrid preconditioner proposed in Section~\ref{sec:multigrid}, using the exact solution, $u_{1ex}$.} To demonstrate the effectiveness of the monolithic multigrid
preconditioner, Table \ref{tab:2D_timings1} presents iteration counts
and CPU times to solution for both the multigrid-preconditioned
FGMRES iterations and the use of a direct solver (MUMPS \cite{amestoy2001}, via the PETSc interface) for the unit square domain, with 
$(u,\vec{v},\vec{\alpha}) \in DG_2(\Omega,\tau_h)\times RT_3(\Omega,\tau_h)\times RT_3(\Omega,\tau_h)$ and
$\partial\Omega = \Gamma_0\cup\Gamma_3$. We note that the iteration counts for monolithic multigrid-preconditioned FGMRES are consistent through all
runs and mesh sizes, and that the scaling of wall-clock time for this approach is
$\mathcal{O}(\revisenew{N})$ or better throughout.  While the direct solver is
faster for small mesh sizes, we see worse than
$\mathcal{O}(\revisenew{N})$ scaling for the wall-clock time with MUMPS at larger mesh sizes, showing the
expected behaviour. Moreover, as we vary the number
of processors over which we parallelize the computation, we see that,
for sufficiently large problems, we have reasonable strong parallel
scalability with the monolithic multigrid solver, although MUMPS is always
faster than our multigrid implementation for this two-dimensional problem.
Table \ref{tab:2D_timings3} presents the case of $\partial\Omega=\Gamma_0\cup\Gamma_2\cup\Gamma_3$, $(u,\vec{v},\vec{\alpha}) \in DG_1(\Omega,\tau_h)\times RT_2(\Omega,\tau_h)\times RT_2(\Omega,\tau_h)$. As we increase the number of processors from 4 to 16, we see better performance with the multigrid solver than the direct solver.  Again, we have good strong parallel scalability with the monolithic multigrid solver, showing \revisenew{3.83x}
speedup for the $1024^2$ mesh, while the direct solver (MUMPS) shows only \revisenew{1.55x} speedup.
\begin{table}
	\centering
	\caption{Wall-clock time (in seconds) and iterations to
		convergence with varying numbers of processors, $p$, for
		monolithic multigrid and a direct solver (MUMPS) for the
		unit square domain with $c_0=0$, $c_1=1$, $\partial\Omega = \Gamma_0\cup\Gamma_3$ and
		$(u,\vec{v},\vec{\alpha}) \in DG_2(\Omega,\tau_h)\times RT_3(\Omega,\tau_h)\times
		RT_3(\Omega,\tau_h)$.} \label{tab:2D_timings1}
	\begin{tabular}{ c| c c c |c c}  
		\toprule
		\multicolumn{1}{c|}
		{${h^{-1}}$} & \multicolumn{3}{c|}{Monolithic}& \multicolumn{2}{c}{MUMPS} \\ \midrule
		&Iterations& Time ($p=1$) & Time ($p=4$) &Time ($p=1$) & Time ($p=4$)\\ \midrule
		$2^6$&5 & 31.48 & 15.45 & 6.07& 4.11\\ 
		$2^7$&5 & 97.30& 37.04& 16.55& 9.35\\ 
		$2^8$&5 & 386.82 & 137.35&74.39 & 39.19\\ 	
		$2^{9}$&5&1553.12 & 623.34&358.06 &183.83 \\ \bottomrule
	\end{tabular}
\end{table}

\begin{table}
	\centering
	\caption{Wall-clock time (in seconds) and iterations to
		convergence with varying numbers of processors, $p$, for
		monolithic multigrid and a direct solver (MUMPS) for the
		unit square domain with $c_0=2$, $c_1=4$, $\partial\Omega = \Gamma_0\cup\Gamma_2\cup\Gamma_3$ and
		$(u,\vec{v},\vec{\alpha}) \in DG_1(\Omega,\tau_h)\times RT_2(\Omega,\tau_h)\times
		RT_2(\Omega,\tau_h)$.} \label{tab:2D_timings3}
	\begin{tabular}{ c| c c c |c c}  
		\toprule
		\multicolumn{1}{c|}
		{$h^{-1}$} & \multicolumn{3}{c|}{Monolithic}& \multicolumn{2}{c}{MUMPS} \\ \midrule
		&Iterations& Time ($p=4$) & Time ($p=16$) &Time ($p=4$) & Time ($p=16$)\\ \midrule
		$2^7$&5 & 19.00 & 10.18 &6.22 & 4.46\\ 
		$2^8$&5 & 48.62 &18.53 & 17.96& 12.24\\ 
		$2^9$&5 & 178.60 & 58.65& 83.58 &53.89 \\ 	
		$2^{10}$&5&854.42 &223.28 & 387.61&250.35 \\ \bottomrule
	\end{tabular}
\end{table}

%\pef{Why do we consider discretisation errors, then Krylov iteration counts, then discretisation errors again below?}
%\ah{ The 2D experiments are organized as follows: Two examples of general $H^2$-elliptic problems; discretisation errors and Krylov iteration counts. Then, we pay more attention to the classical biharmonic operator and Krylov iteration counts for the challenging case(clamped BCs with Nitsche's method). \\ I can change it to the discretization errors, and then Krylov iteration counts if you want.    }
%\pef{If we explain the logic of the order in which we'll do things, I'm OK with it. Scott, what do you think?}
Finally, we consider the classical biharmonic operator with clamped boundary conditions, i.e., $c_0=c_1=0$ and $\partial\Omega=\Gamma_1$ and $(u,\vec{v},\vec{\alpha})\in DG_3(\Omega,\tau_h)\times RT_4(\Omega,\tau_h)\times RT_4(\Omega,\tau_h)$. Table \ref{bih} shows the effectiveness of the monolithic multigrid solver with an $\mathcal{O}(h^{-1})$ weight on the auxiliary operator. Dashes in the table mean that more than 100 iterations were required to converge when the residual norm or its relative reduction is less than $\revisenew{10^{-12}}$. We note that, due to a technical limitation in PCPATCH (where Nitsche boundary terms cannot be treated), these results use an alternate implementation of the star relaxation scheme that is less efficient than PCPATCH.  Consequently, we do not report timings for these experiments, as they are not comparable to the timings reported elsewhere in this paper.
\begin{table}
	\centering
	\caption{Number of iterations to converge with different weights on the auxiliary operator. Here $(u,\vec{v},\vec{\alpha})\in DG_3(\Omega,\tau_{h})\times RT_4(\Omega,\tau_{h})\times RT_4(\Omega,\tau_{h})$. A dash means that convergence was not achieved in 100 iterations.} \label{bih}
	\begin{tabular}{l|*{6}{c}}\toprule
		\backslashbox{$h^{-1}$}{weight}
		&\makebox{1}&\makebox{10}&\makebox{20}
		&\makebox{40}&\makebox{80}&\makebox{$h^{-1}$}\\\midrule
		$2^6$	&29&11 &11 &10 &10 &10 \\
		$2^7$	&- &17 &13 &10 &10 & 10\\
		$2^8$	&- &46 &22 &16 &12 &10  \\
		$2^9$	&- &-  &-  &-& -&  9\\\bottomrule
	\end{tabular}
\end{table}

\subsection{$3D$ experiments}
Here, we consider a test case on the unit cube, with right-hand side and boundary conditions chosen so
that the exact solution is $u_{ex}=\sin(2\pi x)\cos(3\pi y)\sinh(\pi
z)$.  Finite-element convergence is demonstrated in Figure
\ref{3D:square} for $k \in \{0,1,2\}$ with $\partial\Omega = \Gamma_0\cup\Gamma_2\cup\Gamma_3$, with $\Gamma_0$ corresponding to $z=0$ and $z=1$, $\Gamma_2$ corresponding to $y=0$ and $y=1$, and $\Gamma_3$ corresponding to $x=0$ and $x=1$, showing convergence consistent with the analysis of Corollary \ref{thm4}. Table \ref{tab:3D_timings} details the performance of the monolithic multigrid-preconditioned FGMRES solver for $k=0$, compared with a standard direct solver (MUMPS).  We see excellent performance of the monolithic multigrid method, with
iteration counts that are independent of problem size and CPU
time scaling linearly with problem size, and decreasing with
parallelization for sufficiently large problems.  In contrast, we see the
expected rapid growth of required CPU times for MUMPS, and suboptimal
parallel scaling, showing the utility and power of the monolithic
multigrid approach.

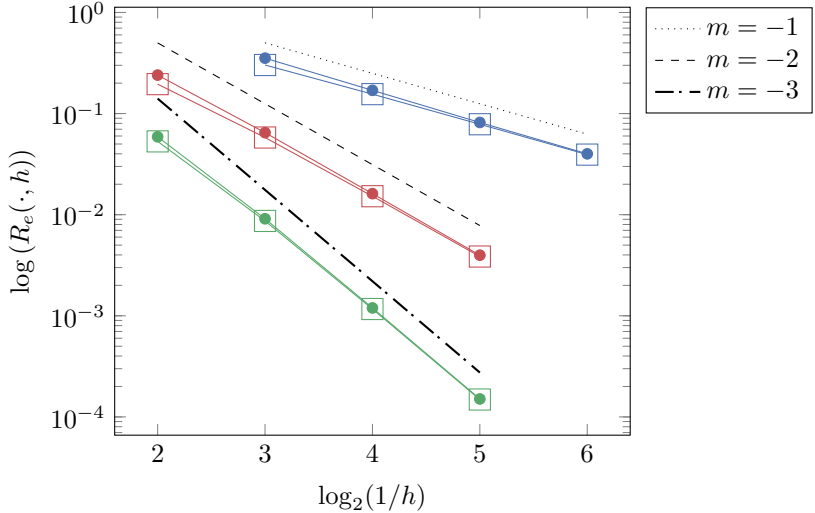
\begin{figure}
	\centering
	
	\begin{tikzpicture}
	\begin{semilogyaxis}[
	xlabel={$\log_2(1/h)$},
	ylabel={$\log\left(R_e(\cdot,h)\right)$},
	legend pos=outer north east
	]
	
	\addplot[
	mark=*,mark options={scale=1,solid},color=seabornblue
	]
	coordinates {

		(3, 0.3532475842379772)
		(4, 0.169724598619751)
		(5, 0.0816565486495772)
		(6, 0.03997958581928016)	
	};
	\addplot[dotted, domain=3:6] {4/(2^x)};
	\addplot[
	mark=square,mark options={scale=2,solid},color=seabornblue
	]
	coordinates {

		(3, 0.3021637748573197)
		(4, 0.15563256338955625)
		(5, 0.07841383140534439)
		(6, 0.03928115131195398)	
	};

	\addplot[
	mark=*,mark options={scale=1,solid},color=seabornred
	]
	coordinates {
		(2, 0.24011465960697295)
		(3, 0.06477039545264777)
		(4, 0.016151393377820025)
		(5, 0.003977249372758478)	
	};
	\addplot[dashed, domain=2:5] {8/(2^x)^2};
	\addplot[
	mark=square,mark options={scale=2,solid},color=seabornred
	]
	coordinates {

		(2, 0.19544623241025652)
		(3, 0.057973618649917215)
		(4, 0.01523050681910914)
		(5, 0.0038570146949304584)	
	};

	\addplot[
	mark=*,mark options={scale=1,solid},color=seaborngreen
	]
	coordinates {

		(2, 0.05895572410965531)
		(3, 0.009109146646074207)
		(4, 0.001196666963350553)
		(5, 0.0001505724139069509)
	};
	\addplot[thick,dash pattern={on 7pt off 2pt on 1pt off 3pt}, domain=2:5] {9/(2^x)^3};
	\addplot[
	mark=square,mark options={scale=2,solid},color=seaborngreen
	]
	coordinates {

		(2, 0.05308727311679876)
		(3, 0.008674745536072165)
		(4, 0.0011672568052638957)
		(5, 0.0001486761018499699)
	};

	\legend
	{,$m=-1$,,,$m=-2$,,,$m=-3$}
	\end{semilogyaxis}
	\end{tikzpicture}
	\caption{Relative approximation errors and rates of convergence
		for the unit cube domain $ \partial\Omega = \Gamma_0\cup\Gamma_2\cup\Gamma_3$, $c_0=4$ and $c_1=2$. Blue, red, and green lines present results for $k=0,1,2$, respectively. \revisenew{Reference lines of slopes -1, -2, and -3 approximate convergence rates.}}\label{3D:square}
\end{figure}
\begin{table}
	\centering
	\caption{Wall-clock time (in seconds) and iterations to
		convergence with varying numbers of processors, $p$, for
		monolithic multigrid and a direct solver (MUMPS) for the
		unit cube domain with $c_0=4$, $c_1=2$, $\partial\Omega = \Gamma_0\cup\Gamma_2\cup\Gamma_3$ and
		$(u,\vec{v},\vec{\alpha}) \in DG_0(\Omega,\tau_h)\times RT_1(\Omega,\tau_h)\times
		RT_1(\Omega,\tau_h)$.} \label{tab:3D_timings}
	\begin{tabular}{ c| c c c |c c}  
		\toprule
		\multicolumn{1}{c|}
		{$h^{-1}$} & \multicolumn{3}{c|}{Monolithic}& \multicolumn{2}{c}{MUMPS} \\ \midrule
		&Iterations& Time ($p=1$) & Time ($p=8$) &Time ($p=1$) & Time ($p=8$)\\ \midrule
		$2^3$&9 &8.05 &5.80 & 2.08& 1.92\\ 
		$2^4$&9 &20.07 &7.34 & 6.80& 2.99\\ 
		$2^5$&9 & 229.04& 37.59&124.05 &39.22 \\ 	
		$2^6$&9& 1847.76 & 276.34& 4903.83& 1148.47\\ \bottomrule
	\end{tabular}
\end{table}

\section{Conclusion}
\label{sec:conclusion}
We consider the mixed finite-element approximation of solutions to
$H^2$-elliptic fourth-order problems, achieved by the transformation of the
fourth-order equation into a system of PDEs.  We find that under natural assumptions on the coefficients of the problem, three
combinations of boundary conditions lead to optimal finite-element convergence. For the fourth case of boundary conditions (``clamped'' boundary conditions, on the solution
and its normal derivative), suboptimal rates of convergence are expected and observed when implemented
using Nitsche's method. While the approach is applicable in both two and three dimensions, we note that it is particularly attractive in 3D, where the cost of conforming methods is prohibitive. It remains an open question whether or not it is possible to employ
alternative approaches (such as adapting
the Nitsche boundary conditions, or the use of alternative penalty
approaches) to regain optimal finite-element convergence for the boundary
conditions where suboptimal convergence is proven and observed here.
\revise{As is common for mixed formulations, the approach we propose herein has higher regularity requirements to guarantee convergence than other possible approaches.  \revisenew{Numerical results suggest that acceptable convergence can be achieved in some cases when this regularity is not present (for a solution in $H^{3}(\Omega)$ but not $H^4(\Omega)$), but in some cases inadmissible solutions are found (with $u\notin H^2(\Omega)$) due to the weaker spaces used.}  Understanding convergence in these cases is a target for future research.}
We additionally propose a
monolithic multigrid algorithm with optimal scaling for the resulting discrete linear
systems. For three-dimensional problems, this approach yields a
preconditioned FGMRES iteration that dramatically outperforms state-of-the-art
direct solvers.

\bibliography{ref}

\end{document}